\documentclass[10pt,a4paper]{article}
\usepackage{amsfonts,amssymb,amsmath,amsthm,mathtools}
\usepackage{aliascnt}
\usepackage{mathrsfs}
\usepackage{enumitem}
\usepackage[margin=25mm]{geometry}
\usepackage[english]{babel}
\usepackage[T1]{fontenc}
\usepackage{lmodern}
\usepackage{microtype}
\usepackage{hyperref}
\usepackage[nameinlink,capitalize]{cleveref}
\usepackage{url}
\hypersetup{
  colorlinks=true,
  linkcolor=blue,
  citecolor=blue,
  urlcolor=blue,
  pdftitle={Quantitative propagation of chaos on arbitrary finite time intervals for the three-dimensional repulsive Vlasov--Poisson--Fokker--Planck equation},
  pdfauthor={Ning Jiang, Juntao Wu},
  pdfsubject={Quantitative propagation of chaos on arbitrary finite time intervals for the unregularized three-dimensional repulsive VPFP equation},
  pdfkeywords={Vlasov--Poisson--Fokker--Planck, Riesz interaction, Coulomb interaction, propagation of chaos, total variation, BBGKY hierarchy, mean-field limit},
  pdfdisplaydoctitle=true
}
\numberwithin{equation}{section}

\newcommand{\R}{\mathbb R}
\newcommand{\T}{\mathbb T}

\newcommand{\Pp}{\mathbb P}
\newcommand{\E}{\mathbb E}
\newcommand{\dd}{\,\mathrm d}
\newcommand{\ind}{\mathbf 1}
\newcommand{\TV}{\mathrm{TV}}
\newcommand{\esssupp}{\operatorname*{ess\,supp}}

\newtheoremstyle{plainnormalnote}%
  {3pt}%
  {3pt}%
  {\itshape}%
  {0pt}%
  {\bfseries}%
  {.}%
  {.5em}%
  {\thmname{#1}\thmnumber{ #2}\thmnote{ \normalfont(#3)}}%
\theoremstyle{plainnormalnote}
\newtheorem{theorem}{Theorem}[section]
\newaliascnt{corollary}{theorem}
\newtheorem{corollary}[corollary]{Corollary}
\aliascntresetthe{corollary}
\newaliascnt{proposition}{theorem}
\newtheorem{proposition}[proposition]{Proposition}
\aliascntresetthe{proposition}
\newaliascnt{lemma}{theorem}
\newtheorem{lemma}[lemma]{Lemma}
\aliascntresetthe{lemma}
\theoremstyle{definition}

\newaliascnt{assumption}{theorem}
\newtheorem{assumption}[assumption]{Assumption}
\aliascntresetthe{assumption}
\newtheoremstyle{boldremark}%
  {3pt}%
  {3pt}%
  {\normalfont}%
  {0pt}%
  {\bfseries}%
  {.}%
  {.5em}%
  {\thmname{#1}\thmnumber{ #2}\thmnote{ \normalfont(#3)}}%
\theoremstyle{boldremark}
\newtheorem{remark}{Remark}[section]
\theoremstyle{plainnormalnote}

\crefname{theorem}{Theorem}{Theorems}
\Crefname{theorem}{Theorem}{Theorems}
\crefname{proposition}{Proposition}{Propositions}
\Crefname{proposition}{Proposition}{Propositions}
\crefname{lemma}{Lemma}{Lemmas}
\Crefname{lemma}{Lemma}{Lemmas}
\crefname{corollary}{Corollary}{Corollaries}
\Crefname{corollary}{Corollary}{Corollaries}
\crefname{definition}{Definition}{Definitions}
\Crefname{definition}{Definition}{Definitions}
\crefname{assumption}{Assumption}{Assumptions}
\Crefname{assumption}{Assumption}{Assumptions}
\crefname{remark}{Remark}{Remarks}
\Crefname{remark}{Remark}{Remarks}
\crefname{equation}{equation}{equations}
\Crefname{equation}{Equation}{Equations}

\begin{document}

\title{\texorpdfstring{\Large\bfseries Quantitative propagation of chaos on arbitrary finite time intervals for the three-dimensional repulsive Vlasov--Poisson--Fokker--Planck equation}{Quantitative propagation of chaos on arbitrary finite time intervals for the three-dimensional repulsive Vlasov--Poisson--Fokker--Planck equation}}
\author{Ning Jiang$^\mathrm{a}$,\quad Juntao Wu$^\mathrm{a}$\\
{\small\itshape $^\mathrm{a}$ School of Mathematics and Statistics, Wuhan University, Wuhan, Hubei 430072, P. R. China}\\[1mm]
{\small Corresponding author: Juntao Wu, \texttt{00036371@whu.edu.cn}}\\[1mm]
}
\date{}

\maketitle

\begin{abstract}
We prove quantitative propagation of chaos in total variation for
repulsive Vlasov--Riesz--Fokker--Planck (VRFP) equations on $\T^d$ and
$\R^d$ with unregularized interaction potentials with singularity
$|x|^{-s}$, $0<s<d/2$. On every finite time interval on which the
limiting classical solution satisfies the stated regularity assumptions,
the particle marginals with a product initial law converge at an algebraic
rate in the particle number, with constants growing at most exponentially
in the marginal order. For the three-dimensional repulsive
Vlasov--Poisson--Fokker--Planck (VPFP) equation on $\R^3$, a weighted
$C^m$ initial probability density, $m\ge3$, with Gaussian velocity decay
and polynomial spatial decay of order $M>3$ gives convergence to the
global classical solution on every finite time interval. The proof
combines conditioning of the initial product law with weighted BBGKY
estimates. A decomposition of the potential and a bounded modification
of the far-field force yield weighted bounds growing geometrically in
the marginal order on the prescribed interval. A centered second-moment
estimate permits summation of the error hierarchy, and a synchronous
coupling transfers the estimate to the original dynamics.
\end{abstract}



\section{Introduction}\label{sec:intro}

We study quantitative propagation of chaos for stochastic particle systems
with unregularized repulsive Riesz interactions. Our general result,
\Cref{thm:fixed-time-riesz}, gives an algebraic rate of convergence in
total variation on every prescribed finite interval under regularity
assumptions on the limiting Vlasov--Riesz--Fokker--Planck (VRFP)
solution. For the three-dimensional
Coulomb interaction, the global classical solution theory yields
\Cref{cor:global-coulomb} on every finite interval for weighted smooth
initial data.

\subsection{The models}\label{subsec:intro-model}

Let $d\ge2$, $0<s<d-1$, and $\mathcal D\in\{\T^d,\R^d\}$, where
$\T^d=(\R/\mathbb Z)^d$ has unit volume.  The phase space is
$\Omega_{\mathcal D}=\mathcal D\times\R^d$.
Fix $\sigma>0$ and $c_{d,s}>0$.  On $\R^d$, set
\begin{equation}\label{eq:intro-whole-riesz-potential}
 \phi_{\R,s}(x)=1+c_{d,s}|x|^{-s},\qquad
 K_{\R,s}(x)=-\nabla\phi_{\R,s}(x)
 =c_{d,s}s\frac{x}{|x|^{s+2}}.
\end{equation}
On the torus, let $\phi_{\T,s}\ge1$ be even and smooth away from the
origin, with lower semicontinuous extension $\phi_{\T,s}(0)=+\infty$.
In a coordinate neighborhood of the origin, assume that
\begin{equation}\label{eq:intro-periodic-riesz-potential}
 \phi_{\T,s}(x)=c_{d,s}|x|^{-s}+H_{\T,s}(x),\qquad
 H_{\T,s}\in C^\infty.
\end{equation}
Set $K_{\T,s}=-\nabla\phi_{\T,s}$ away from the origin.  Then
$K_{\T,s}(x)=c_{d,s}s x/|x|^{s+2}+O(1)$ near zero.  This class includes
smooth periodic extensions of the Euclidean singularity and, after an
additive normalization, positive multiples of the periodic Riesz Green kernel; see
\cite[equation~(3.1)]{ChodronRosenzweigSerfaty2025}.  Additive constants in the potential
leave the force unchanged.

We consider the second-order stochastic particle system
\begin{equation}\label{eq:intro-riesz-particle}
\left\{
\begin{aligned} &\dd X_i=V_i\,\dd t,\\
 &\dd V_i=\frac1N\sum_{j\ne i}K_{\mathcal D,s}(X_i-X_j)\,\dd t
 +\sqrt{2\sigma}\,\dd B_i,
 \qquad i=1,\ldots,N,
  \end{aligned}
  \right.
\end{equation}
where $B_1,\ldots,B_N$ are independent standard $\R^d$-valued Brownian
motions, independent of the initial configuration.  Positions are understood
modulo the torus when $\mathcal D=\T^d$.  The corresponding mean-field equation
is the VRFP equation
\begin{equation}\label{eq:intro-riesz-vpfp}
 \left\{
 \begin{aligned}
  &\partial_t f+v\cdot\nabla_x f+(K_{\mathcal D,s}*\rho_f)\cdot\nabla_v f
  =\sigma\Delta_v f,\\
  &\rho_f(t,x)=\int_{\R^d}f(t,x,v)\,\dd v,\qquad f(0)=f_0.
 \end{aligned}
 \right.
\end{equation}

For $\mathcal D=\R^3$, $s=1$ and $c_{3,1}=c>0$, the force is
$K_{\R,1}(x)=cx/|x|^3$, and \eqref{eq:intro-riesz-vpfp} is the
three-dimensional repulsive Vlasov--Poisson--Fokker--Planck (VPFP)
equation.

For each $N$, \Cref{prop:common-finite-N} provides a pathwise unique global
strong solution of \eqref{eq:intro-riesz-particle} whose particle
positions remain distinct, for every initial law supported on configurations
with distinct positions and independent of the Brownian motions. This applies to all
initial densities below.

\paragraph{Notation.}\label{par:intro-notation}
Let $F_N(t)$ denote the law of \eqref{eq:intro-riesz-particle} and let
$F_{N,k}(t)$ be its $k$-particle marginal.  We write
$z_i=(x_i,v_i)$, $Z_k=(z_1,\ldots,z_k)$,
$X_k=(x_1,\ldots,x_k)$, $V_k=(v_1,\ldots,v_k)$ and
$\Delta_{V_k}=\sum_{i=1}^k\Delta_{v_i}$.
We also write $f_t=f(t,\cdot)$ and
$f_t^{\otimes k}(Z_k)=\prod_{i=1}^k f(t,z_i)$.  For any integrable phase-space
function $g$, write $\rho_g(x):=\int_{\R^d}g(x,v)\,\dd v$ and
$\rho_0:=\rho_{f_0}$.

We identify absolutely continuous probability measures with their
densities. For probability measures on a measurable space $S$, we use
$\|\mu-\nu\|_{\TV}=|\mu-\nu|(S)$, which is the $L^1$ distance for
densities. We write $\langle x\rangle=(1+|x|^2)^{1/2}$.

For a Markov kernel $P$, a measurable map $\Pi$, and a coupling
$(X,Y)$ of $\mu,\nu$, we have
\begin{equation}\label{eq:intro-TV-contraction}
 \|\Pi_\#P\mu-\Pi_\#P\nu\|_{\TV}\le\|\mu-\nu\|_{\TV}\le2,
 \qquad \|\mu-\nu\|_{\TV}\le2\Pp(X\ne Y).
\end{equation}
We write $q'=q/(q-1)$ and, on a fixed interval $[0,T]$,
$C_tX=C([0,T];X)$ and $L_t^pX=L^p(0,T;X)$.
We write $d_{\mathcal D}$ for the Euclidean distance on $\R^d$ and the
flat distance on $\T^d$.
The space $C_b^r$ consists of functions with bounded continuous derivatives
up to order $r$. Spatial and velocity multi-indices have nonnegative integer
entries. We denote the unweighted $L^p$ norm by $\|\cdot\|_p$.
For $1\le q<\infty$ and a positive weight $W$ on
$\Omega_{\mathcal D}^k$, we use
\[
 \|h\|_{L^q(W)}^q=\int_{\Omega_{\mathcal D}^k}|h(Z_k)|^qW(Z_k)\,\dd Z_k.
\]
We write $\dd Z_{a:b}=\prod_{j=a}^b\dd x_j\dd v_j$; an integral over
an empty set of variables leaves its integrand unchanged. Unless specified
otherwise, phase-space integrals are over $\Omega_{\mathcal D}$ in each
displayed variable. Pointwise density bounds hold almost everywhere. Constants denoted by
$C$ may change from line to line; their parameter dependence is specified
in each section.

For an integer $m\ge0$, $a>0$ and $M>d$, we set
\begin{equation*}
 C^m_{M,a}(\R^{2d})
 :=\{h\in C^m(\R^{2d}):\|h\|_{C^m_{M,a}}<\infty\},
\end{equation*}
with
\begin{equation}\label{eq:whole-space-classical-data-norm}
 \|h\|_{C^m_{M,a}}
 :=\max_{|\xi|+|\eta|\le m}\sup_{x,v}
 \langle x\rangle^M e^{a|v|^2}
 |\partial_x^\xi\partial_v^\eta h(x,v)|.
\end{equation}

On the torus, for an integer $m\ge0$ and $a>0$, we use
\begin{equation*}
 C^m_{\T,a}
 :=\{h\in C^m(\T^d\times\R^d):\|h\|_{C^m_{\T,a}}<\infty\},
 \quad
 \|h\|_{C^m_{\T,a}}
 :=\max_{|\xi|+|\eta|\le m}\sup_{x,v}
 e^{a|v|^2}|\partial_x^\xi\partial_v^\eta h(x,v)|.
\end{equation*}

A \emph{classical solution} on $[0,T]$ is a function
$f\in C([0,T];L^1)$ with $f(0)=f_0$, whose derivatives
$\partial_tf$, $\nabla_xf$, $\nabla_vf$ and $D_v^2f$ are continuous
on $(0,T]\times\Omega_{\mathcal D}$ and which satisfies
\eqref{eq:intro-riesz-vpfp} pointwise there.

\subsection{Main results}\label{subsec:intro-main-results}

We first state the assumptions on the limiting VRFP solution and the
quantitative comparison for Riesz interactions.

\begin{assumption}[Limiting solution]\label{ass:limiting-solution}
Let $d\ge2$, $0<s<d-1$, $\sigma>0$, $T>0$, and
$\mathcal D\in\{\T^d,\R^d\}$. On $[0,T]$, equation
\eqref{eq:intro-riesz-vpfp} has a nonnegative classical solution $f$
with initial datum $f_0$ and unit mass, satisfying
\begin{equation}\label{eq:intro-field-regularity}
 E_f=K_{\mathcal D,s}*\rho_f\in C([0,T];C_b^1),
\end{equation}
and
\begin{equation}\label{eq:intro-velocity-derivative}
 \int_0^T\|\nabla_v f(t)\|_1\,\dd t<\infty.
\end{equation}
For some $a_0>0$ and $C_0<\infty$, the initial datum satisfies
\begin{equation}\label{eq:intro-initial-decay}
 0\le f_0(x,v)\le
 \begin{cases}
 C_0e^{-a_0|v|^2},&\mathcal D=\T^d,\\
 C_0\langle x\rangle^{-M}e^{-a_0|v|^2},
       &\mathcal D=\R^d,\quad M>d.
 \end{cases}
\end{equation}
\end{assumption}

The comparison estimates use an exponent $q$ satisfying
\begin{equation}\label{eq:intro-q-range}
 2\le q<\infty,\qquad q>\frac d{d-s-1}.
\end{equation}

\begin{theorem}
\label{thm:fixed-time-riesz}\label{thm:quantitative-comparison}
Suppose that \Cref{ass:limiting-solution} holds, $0<s<d/2$, and
$q$ satisfies \eqref{eq:intro-q-range}.
Let $F_N$ be the law of \eqref{eq:intro-riesz-particle} with initial
distribution $f_0^{\otimes N}$. There exist
$C_T,c_T\ge1$ such that, for
$\gamma_T=q^{-1}e^{-c_TT}$ and all $N\ge1$, $1\le k\le N$, we have
\begin{equation}\label{eq:intro-fixed-time-TV}
 \sup_{0\le t\le T}\|F_{N,k}(t)-f_t^{\otimes k}\|_{\TV}
 \le C_T^k\bigl(N^{-\gamma_T}+N^{-(d/s-2)}\bigr).
\end{equation}
The constants are independent of $N,k$ and depend only on
$d,s,q,\sigma,T$, the domain, the potential, and the bounds in
\Cref{ass:limiting-solution}.
\end{theorem}

\begin{remark}\label{rem:main-exponents}
The constants $C_T,c_T$ depend on $T$ and the bounds in
\Cref{ass:limiting-solution} over $[0,T]$; the choice of $c_T$
is given in \eqref{eq:ft-final-parameters}.
The restriction $s<d/2$ in the product-data theorem enters through
the conditioning estimate \eqref{eq:common-riesz-conditioning-cost}.
The force condition is
\[
 q>\frac d{d-s-1}\quad\Longleftrightarrow\quad(s+1)q'<d.
\]
Thus $K_{\mathcal D,s}\in L^{q'}_{\mathrm{loc}}$ near the origin.
\end{remark}

\begin{remark}
\label{rem:intro-classical-solutions}
Under the parameter restrictions of \Cref{thm:fixed-time-riesz},
let $m\ge3$ and $a_0>0$. Nonnegative unit-mass data in
$C^m_{\T,a_0}$, or in $C^m_{M,a_0}$ with $M>d$ on $\R^d$,
satisfy \Cref{ass:limiting-solution} on a local interval by
\cite[Proposition~B.1]{JiangWuShortTime}, applied with velocity
exponent $a_0/2$ and spatial exponent $(M+d)/2$ in the whole space.
\end{remark}

For the three-dimensional Coulomb interaction, the comparison theorem
gives the following result for weighted smooth initial data.

\begin{theorem}[Three-dimensional VPFP on every finite interval]
\label{cor:global-coulomb}
Let $\mathcal D=\R^3$, $s=1$, $c_{3,1}=c>0$, and $\sigma>0$.
Let $m\ge3$ be an integer, $M>3$, and $a_0>0$.
Suppose that $f_0\in C^m_{M,a_0}(\R^6)$ is nonnegative and has unit mass.
Let $f$ be the global classical solution of
\eqref{eq:intro-riesz-vpfp}, and let $F_N$ be the particle law with
initial density $f_0^{\otimes N}$. For every $0<T<\infty$ and every
$q>3$, there exist $C_T,c_T\ge1$, independent of $N,k$, such that
\begin{equation}\label{eq:intro-global-coulomb}
 \sup_{0\le t\le T}\|F_{N,k}(t)-f_t^{\otimes k}\|_{\TV}
 \le C_T^kN^{-\gamma_T},
 \qquad \gamma_T=q^{-1}e^{-c_TT},
 \qquad N\ge1,\quad 1\le k\le N.
\end{equation}
The constants depend on $T,q,\sigma,c$ and $f_0$.
\end{theorem}

The global classical solution theory of \cite{Bouchut1993}, in the
formulation of \cite[Theorem~1.1 and Remark~1.1]{HuangLiuPickl2020},
gives the bounds in \Cref{ass:limiting-solution} on every finite
interval under the hypotheses of \Cref{cor:global-coulomb}.
The verification is given at the end of
Section~\ref{sec:fixed-time-riesz}.

%\begin{remark}\label{rem:intro-growing-k}
%By \eqref{eq:intro-fixed-time-TV}, convergence in total variation is
%uniform on $[0,T]$ for every sequence $1\le k_N\le N$ satisfying
%\[
% k_N\log C_T-\min\{\gamma_T,d/s-2\}\log N\longrightarrow-\infty.
%\]
%For $C_T>1$, this includes $k_N\le c\log N$ with
%$0<c<\min\{\gamma_T,d/s-2\}/\log C_T$.
%\end{remark}

\paragraph{Relation to previous work.}\label{subsec:intro-comparison}

For Lipschitz interaction forces, Braun and Hepp
\cite{BraunHepp1977} established the mean-field limit for classical
particle systems, and Dobrushin \cite{Dobrushin1979} proved stability
estimates in Wasserstein distance. McKean \cite{McKean1966} related
nonlinear parabolic equations to nonlinear Markov processes, and
Sznitman \cite{Sznitman1991} developed the probabilistic theory of
propagation of chaos. The surveys of Golse \cite{Golse2016} and
Jabin \cite{Jabin2014} discuss mean-field limits for Vlasov equations;
Chaintron and Diez \cite{ChaintronDiez2022I,ChaintronDiez2022II}
give a broader account of the models, methods and applications.

For deterministic second-order systems, Hauray and Jabin
\cite{HaurayJabin2007,HaurayJabin2015} treated mildly singular forces
of the form $|K(x)|\le C|x|^{-\alpha}$ with $\alpha<1$, under
corresponding assumptions on the derivatives of the force and on the
initial configurations. More singular forces have been treated with
an $N$-dependent cutoff. Boers and Pickl \cite{BoersPickl2016}
compared particle trajectories with the mean-field characteristics
for forces less singular than Coulomb in three dimensions, with a
cutoff at scale $N^{-1/3}$. Lazarovici and Pickl
\cite{LazaroviciPickl2017} derived the three-dimensional
Vlasov--Poisson equation for Coulomb and Newton interactions with a
cutoff at scale $N^{-\delta}$, $0<\delta<1/3$.
For bounded forces, Jabin and Wang \cite{JabinWang2016} used relative
entropy to prove propagation of chaos for Vlasov systems, including
systems with diffusion in velocity.

For the VPFP equation, Hauray and Salem \cite{HauraySalem2019}
proved quantitative propagation of chaos in one dimension without
regularizing the interaction. Carrillo, Choi and Salem
\cite{CarrilloChoiSalem2019} obtained propagation of chaos for
Coulomb interactions with a polynomial cutoff. In three dimensions,
Huang, Liu and Pickl \cite{HuangLiuPickl2020} obtained quantitative
estimates by comparing the stochastic particle trajectories with
independent trajectories driven by the limiting field.
Chen, Jung, Pickl and Wang \cite{ChenJungPicklWang2026} combined
such trajectory estimates with relative entropy to prove
$L^1$ convergence of the marginals for regularized Coulomb systems,
both with and without velocity diffusion.

For first-order systems, the modulated energy method gives mean-field
limits for interactions with Coulomb and Riesz singularities.
Serfaty \cite{Serfaty2020} treated Coulomb and super-Coulombic Riesz
interactions, and Nguyen, Rosenzweig and Serfaty
\cite{NguyenRosenzweigSerfaty2022} extended the analysis to Riesz-type
potentials with $0<s<d$. Rosenzweig and Serfaty
\cite{RosenzweigSerfaty2023} obtained global-in-time convergence for
diffusive Riesz-type flows in the range $0<s<d-2$.

Relative entropy provides another approach to singular first-order
systems. Jabin and Wang \cite{JabinWang2018} proved quantitative
propagation of chaos for stochastic systems with
$W^{-1,\infty}$ kernels, including the two-dimensional viscous vortex
model. Bresch, Jabin and Wang \cite{BreschJabinWang2020} combined
relative entropy and modulated potential energy in a modulated free
energy and obtained quantitative estimates for repulsive singular
interactions with diffusion. Their work
\cite{BreschJabinWang2023} treats singular attractive interactions,
including the subcritical Patlak--Keller--Segel model.
Chodron de Courcel, Rosenzweig and Serfaty
\cite{ChodronRosenzweigSerfaty2025} used modulated free energy and
decay estimates for the limiting equation to prove sharp
uniform-in-time convergence for periodic Riesz gradient flows.

Estimates at the level of individual marginals have also been developed
through hierarchies of differential inequalities. Lacker
\cite{Lacker2023} obtained relative entropy bounds of order
$(k/N)^2$ under a functional inequality for the limiting law.
Hess-Childs and Rowan \cite{HessChildsRowan2025} proved quantitative
estimates in a weighted $L^2$ distance and higher-order expansions
for diffusions with bounded interactions. Songbo Wang
\cite{Wang2026Local} treated singular $W^{-1,\infty}$ interactions
using hierarchies of relative entropies and Fisher informations, and
of $\chi^2$ divergences and Dirichlet energies.
B\'ejar-L\'opez, Blaustein, Jabin and Soler
\cite{BejarBlausteinJabinSoler2026} obtained bounds uniform in time
and in the particle number for the $L^2$ norms of marginals of
singular first-order systems, including initial laws that are not chaotic.

For singular kinetic systems, Bresch, Jabin and Soler
\cite{BJS2025} introduced exponentially weighted $L^q$ estimates
for the BBGKY hierarchy, with weights defined by the internal energy
of the particles in each marginal. Velocity dissipation controls the
velocity derivative in the interaction term, while a weighted
H\"older inequality bounds its flux by the next marginal.
Their Theorem~2 gives short-time weak convergence on the torus for
$K\in L^p$, $p>1$, including the two-dimensional repulsive Coulomb
interaction. It assumes exponential integrability of the potential,
Gaussian velocity-moment bounds for entropy solutions of the
$N$-particle Fokker--Planck equation,
and initially chaotic marginals with $L^\infty$ bounds growing
geometrically in their order. Under the additional assumptions
$K\in L^2$ and quantitative weighted initial bounds, their
Theorem~3 gives quantitative convergence in weighted $L^2$.

Bresch, Duerinckx and Jabin \cite{BreschDuerinckxJabin2024}
developed a duality method based on linearized correlations for
square-integrable forces, with or without diffusion.
For forces in $H^r_{\mathrm{loc}}$ with $r>0$, their Theorem~2
gives total variation estimates of order $N^{-C^{-1}e^{-Ct}}$
under regularity assumptions on the force away from the origin,
the limiting field and weighted derivatives of the limiting solution.
Khoury and Jabin \cite{KhouryJabin2026} obtained quantitative
estimates for marginals and correlation functions through a hierarchy
of dual cumulants. Duerinckx and Jabin
\cite{DuerinckxJabin2026} used kinetic regularization in the dual
hierarchy to derive the unregularized two-dimensional
Vlasov--Poisson equation, with or without velocity diffusion,
under regularity assumptions on the limiting solution.
For kinetic McKean--Vlasov systems with interactions in a kinetic
Kato class, Hao, Zhang and Zhao \cite{HaoZhangZhao2026}
obtained quantitative relative entropy estimates on path space
and a central limit theorem.

Feng and Wang \cite{FengWang2026} combined modulated Gibbs
measures with the partition-function estimates of Wang and Zhao
\cite{WangZhao2026} to apply the weighted BBGKY estimates to
product initial laws. Their Theorem~1.3 proves weak propagation
of chaos for repulsive Riesz potentials with $0<s<d/2$ on an
$N$-independent short-time interval, including the unregularized
three-dimensional VPFP equation. The result assumes a smooth
limiting solution and smooth bounded initial data with bounded
spatial density, together with a Gaussian velocity moment on the
torus and a Gaussian moment in $(x,v)$ in the whole space.

The regularity of the limiting equation is supplied by the classical
solution theory for VPFP. Victory and O'Dwyer
\cite{VictoryODwyer1990} studied classical solutions, Bouchut
\cite{Bouchut1993} proved global existence and uniqueness of smooth
solutions in three dimensions, and Ono and Strauss
\cite{OnoStrauss2000} obtained further results on regular solutions.
For Riesz interactions, Choi and Jeong \cite{ChoiJeong2024}
studied well-posedness and singularity formation for the
Vlasov--Riesz system. Wang and Zhang \cite{WangZhang2025}
proved local well-posedness for VRFP in modulation spaces, while
Choi, Jeong and Kang \cite{ChoiJeongKang2024} constructed global
solutions near a Maxwellian for VRFP with velocity friction.

In our previous work \cite{JiangWuShortTime}, conditioning of
the initial product law and auxiliary $k$-particle Fokker--Planck
equations were combined with weighted BBGKY estimates to obtain
quantitative propagation of chaos for unregularized repulsive
Riesz interactions. Under the hypotheses of
\Cref{thm:fixed-time-riesz}, \cite[Theorem~1.2]{JiangWuShortTime}
gives
$0<\tau\le T$ and $C\ge1$, independent of $N,k$, such that
\[
 \sup_{0\le t\le\tau}\|F_{N,k}(t)-f_t^{\otimes k}\|_{\TV}
 \le C^k\bigl(N^{-1/q}+N^{-(d/s-2)}\bigr).
\]
For the three-dimensional Coulomb interaction, this yields every
rate $N^{-b}$ with $0<b<1/3$ on a corresponding short-time interval.
The whole-space estimates use polynomial spatial weights and
Gaussian velocity decay.

The present paper extends this quantitative comparison to every
prescribed finite interval on which
\Cref{ass:limiting-solution} holds. Estimate
\eqref{eq:intro-fixed-time-TV} replaces $N^{-1/q}$ by
$N^{-\gamma_T}$, with $\gamma_T=q^{-1}e^{-c_TT}>0$, and retains
the conditioning error $N^{-(d/s-2)}$.
The conditioning and auxiliary estimates from
\cite{JiangWuShortTime} are recalled in
Sections~\ref{sec:conditioning}--\ref{sec:common-comparison}.
Section~\ref{sec:fixed-time-riesz} establishes weighted marginal
bounds on the prescribed interval and a centered second-moment
estimate for the bounded modification of the force. Summation of
the resulting error hierarchy and synchronous coupling then give
the comparison for the original dynamics. In three dimensions,
the global classical VPFP theory yields
\Cref{cor:global-coulomb} on every finite interval for the weighted
smooth initial data stated there.

\paragraph{Finite-time estimates and proof strategy.}
\label{subsec:intro-finite-time}\label{subsec:intro-proof-outline}

Fix $T>0$. The particle law at a positive time is generally correlated,
whereas \cite[Theorem~1.2]{JiangWuShortTime} starts from a product law. For the evolution from the conditioned initial law,
\eqref{eq:TV-GN-bound} yields weighted marginal bounds of order $C_T^N$.
We obtain bounds for the modified system that grow geometrically in
the marginal order, with a base independent of $N$.

From the decomposition \eqref{eq:ft-potential-split} and
\eqref{eq:ft-small-near-force}, we obtain
\[
 K=\widetilde K^a+\overline K^a,\qquad
 \|\widetilde K^a\|_{L^{q'}}^q\le Ca^\kappa,\qquad
 \kappa=q(d-s-1)-d>0.
\]
For the modified law $\widetilde G_N$ in
\eqref{eq:ft-clipped-system}, \Cref{lem:ft-global-cap} uses this
smallness to choose $a=a_T>0$ such that
\[
 \sup_{0\le t\le T}\int\widetilde G_{N,k}(t)^q
                            \widetilde W_{N,k}(t)
 \le R_T^k,\qquad 1\le k\le N.
\]
Here $a_T$ and $R_T$ are independent of $N$, and the weights are
defined in \eqref{eq:ft-weights}.

We outline the proof of \Cref{thm:fixed-time-riesz}.
Let $G_N$ be the law of the original particle system with initial
density $G_N^0$, take $a=a_T$ in the modified system, and let
$U_{N,k}$ solve \eqref{eq:TV-U-PDE}.
We use the decomposition
\begin{equation}\label{eq:intro-error-decomposition}
 \begin{aligned}
 F_{N,k}-f_t^{\otimes k}
 ={}&(F_{N,k}-G_{N,k})+(G_{N,k}-\widetilde G_{N,k})\\
 &+(\widetilde G_{N,k}-U_{N,k})+(U_{N,k}-f_t^{\otimes k}).
 \end{aligned}
\end{equation}

\begin{enumerate}[label=\textup{(\roman*)}]
\item \emph{Initial conditioning.}
By total variation contraction and
\eqref{eq:common-riesz-conditioning-cost}, we have
\[
 \sup_{t\le T}\|F_{N,k}(t)-G_{N,k}(t)\|_{\TV}
 \le\|f_0^{\otimes N}-G_N^0\|_{\TV}
 \le C_TN^{2-d/s}.
\]

\item \emph{Comparison with the auxiliary equations.}
Combining \eqref{eq:common-riesz-conditioning-marginals} with
\Cref{lem:TV-U-product} yields
\[
 \sup_{t\le T}\|U_{N,k}(t)-f_t^{\otimes k}\|_1
 \le C_RkN^{1-d/s}+\frac{C_Tk^{2-1/q}}N
 \le\frac{C_T^k}{N},
\]
since $d/s>2$.

\item \emph{The error hierarchy.}
The weighted errors defined in \eqref{eq:ft-error-definition} satisfy
\[
 Y_{N,k}(t)=\int|\widetilde G_{N,k}(t)-U_{N,k}(t)|^q
                       \widetilde W_{N,k}(t),
 \qquad Y_{N,k}(0)=0.
\]
Equations \eqref{eq:ft-near-cap} and \eqref{eq:ft-U-cap} imply
$Y_{N,k}\le C_4^k$, with $C_4$ defined in \eqref{eq:ft-error-cap}.
Combining the auxiliary remainder estimates and the
centered second-moment identity \eqref{eq:ft-exact-variance}, we obtain
\[
 \begin{aligned}
 Y_{N,k}(t)-Y_{N,k}(s)
 \le C_T\int_s^t\Bigl[&k\bigl((1+M_{a_T})^qY_{N,k}
             +Y_{N,k+1}+(1+M_{a_T})^{2q}Y_{N,k+2}\bigr)\\
             &+(1+M_{a_T})^{2q}D_T^k\frac2N\Bigr](u)\dd u,
 \end{aligned}
\]
where $M_{a_T}=\|\overline K^{a_T}\|_\infty$ and
$Y_{N,N+1}=Y_{N,N+2}=0$; see \Cref{lem:ft-error-hierarchy}.
Geometric summation and Bihari's inequality in \Cref{lem:ft-osgood},
followed by the weighted H\"older inequality, yield
\begin{align*}
 \sup_{t\le T}Y_{N,k}(t)&\le C_T^kN^{-q\gamma_T},\\
 \sup_{t\le T}\|\widetilde G_{N,k}(t)-U_{N,k}(t)\|_1
 &\le C_T^kN^{-\gamma_T},
 \qquad q\gamma_T=e^{-c_TT},
\end{align*}
with $c_T$ defined in \eqref{eq:ft-final-parameters}.

\item \emph{Synchronous coupling.}
Couple the original and modified trajectories $Z_N,\widetilde Z_N$
with the same initial configuration and Brownian motions.
\Cref{lem:ft-spatial-observables,lem:ft-unclipping} imply
\begin{align*}
 \sup_{t\le T}\|G_{N,k}(t)-\widetilde G_{N,k}(t)\|_{\TV}
 \le2\Pp\bigl(Z_N|_{[0,T]}\ne\widetilde Z_N|_{[0,T]}\bigr)\le C_TN^{-\gamma_T}.
\end{align*}
\end{enumerate}
Inserting these estimates into \eqref{eq:intro-error-decomposition}
proves \eqref{eq:intro-fixed-time-TV}. The global classical solution
theory then yields \Cref{cor:global-coulomb} for the three-dimensional
Coulomb interaction.

Sections~\ref{sec:conditioning}--\ref{sec:common-comparison}
state the conditioning and auxiliary estimates used in the proof.
Section~\ref{sec:fixed-time-riesz} proves the two main theorems.
Section~\ref{sec:kernel-verification} and
Appendix~\ref{app:fixed-N-passage} construct the particle dynamics
and justify the approximation arguments.

\section{Conditioning of the initial product measure}\label{sec:conditioning}
For $1\le k\le N$, define the $k$-particle energy
\begin{equation}\label{eq:comparison-energy}
	e_{N,k}(Z_k)=\sum_{i=1}^k(1+|v_i|^2)
	+\frac1N\sum_{i\ne j\le k}\phi_{\mathcal D,s}(x_i-x_j).
\end{equation}

We recall the conditioning result of
\cite[Corollary~2.3]{JiangWuShortTime}, with the potential and total
variation convention of Section~\ref{subsec:intro-model}.
We denote the event $\mathcal G_N^\sharp(R)$ constructed there by
$\mathcal G_N(R)$. For all sufficiently large $N$, this event bounds
the sum of the pair potentials involving each particle by $RN$ and
each pair potential by a fixed multiple of $N$. For the remaining
particle numbers, only the first condition is imposed; see
\cite[Proposition~2.2]{JiangWuShortTime}.

\begin{corollary}[{\cite[Corollary~2.3]{JiangWuShortTime}}]
\label{prop:common-riesz-conditioning}
Let $d\ge2$, $0<s<d/2$, and let $f_0$ be a probability density on
$\Omega_{\mathcal D}$ with $\rho_0\in L^\infty(\mathcal D)$.
There exists $R_1>0$ such that, for every $R\ge R_1$ and every $N\ge1$,
there is a symmetric closed event $\mathcal G_N(R)$ in the
position variables. Define
\[
 G_N^0=p_N^{-1}\ind_{\mathcal G_N(R)}f_0^{\otimes N},
 \qquad p_N=f_0^{\otimes N}(\mathcal G_N(R)),\qquad \eta_N=1-p_N.
\]
For $1\le k\le N$, we have
\begin{gather}
 p_N\ge\tfrac12,\qquad \eta_N\le C_RN^{2-d/s},\qquad
 \|G_N^0-f_0^{\otimes N}\|_{\TV}=2\eta_N,
 \label{eq:common-riesz-conditioning-cost}\\
 G_{N,k}^0\le2f_0^{\otimes k},\qquad
 \|G_{N,k}^0-f_0^{\otimes k}\|_1\le C_RkN^{1-d/s},
 \label{eq:common-riesz-conditioning-marginals}\\
 \frac1N\sum_{i\ne j\le k}\phi_{\mathcal D,s}(x_i-x_j)\le Rk
 \quad\text{on }\esssupp G_{N,k}^0.\notag
\end{gather}
The constants are independent of $N,k$.
\end{corollary}

The initial bound \eqref{eq:intro-initial-decay} implies
\[
 \|\rho_0\|_\infty
 \le C_0\int_{\R^d}e^{-a_0|v|^2}\,\dd v
 =C_0(\pi/a_0)^{d/2}.
\]
By \eqref{eq:intro-TV-contraction} and
\eqref{eq:common-riesz-conditioning-cost}, the particle path laws
started from $G_N^0$ and $f_0^{\otimes N}$ differ in total variation
on $C([0,T];\Omega_{\mathcal D}^N)$ by at most $2C_RN^{2-d/s}$.

By \cite[equation~(2.19)]{JiangWuShortTime}, for every
$0<\beta<a_0$, the conditioned marginals satisfy
\begin{equation}\label{eq:common-conditioned-initial-exponential}
 G_{N,k}^0\le A_\beta^ke^{-\beta e_{N,k}}
                 \prod_{i=1}^k\langle x_i\rangle^{-M},
 \qquad A_\beta:=\max\{1,2C_0e^{\beta(1+R)}\}.
\end{equation}
The spatial factors are omitted on $\T^d$.
Moreover, from \eqref{eq:common-riesz-conditioning-marginals} and $d/s>2$,
we obtain
\begin{equation}\label{eq:conditioned-initial-error}
 \|G_{N,k}^0-f_0^{\otimes k}\|_1
 \le C_RkN^{1-d/s}
 \le \frac{\max\{2,2C_R\}^k}{N}.
\end{equation}
Throughout the proof, fix $R=R_1$, choose $0<\beta<a_0$, and set
$A=\max\{A_\beta,2,2C_R\}$. We use this conditioned density $G_N^0$
in the comparison with the limiting equation.

\begin{remark}[Exponential interaction weights]
\label{rem:product-exponential-divergence}
By \cite[Remark~2.1]{JiangWuShortTime}, for continuous nonzero $f_0$,
$N\ge k\ge2$, $q\ge1$ and $\lambda>0$, we have
\[
 \int(f_0^{\otimes k})^q e^{\lambda e_{N,k}}\,\dd Z_k=+\infty.
\]
\end{remark}

\section{Auxiliary equations and weighted estimates}\label{sec:common-comparison}

Assume the hypotheses of \Cref{thm:fixed-time-riesz}, and let
$G_N^0$ be the conditioned density from Section~\ref{sec:conditioning}.
In particular, $0<s<d/2$, $q$ satisfies \eqref{eq:intro-q-range},
and \Cref{ass:limiting-solution} holds on $[0,T]$.
Write $\phi,K$ for the potential and force, and let $e_{N,k}$ be
given by \eqref{eq:comparison-energy}. Set
\begin{equation}\label{eq:comparison-error-scale}
 \varepsilon_N=2/N,\qquad
 \beta_2=\min\{\beta,a_0/2\}.
\end{equation}
Equation \eqref{eq:common-conditioned-initial-exponential} implies
\begin{equation}\label{eq:comparison-initial-envelope}
 G_{N,k}^0\le A^ke^{-\beta_2e_{N,k}}
                  \prod_{i=1}^k\langle x_i\rangle^{-M}.
\end{equation}
The spatial factors are omitted on $\T^d$, and the initial error is
bounded by $A^k\varepsilon_N$ by \eqref{eq:conditioned-initial-error}.
At $k=N$, \eqref{eq:comparison-initial-envelope} ensures finite expected
energy. Since $G_N^0$ is a density, it assigns zero mass to the collision
diagonals. Thus \Cref{prop:common-finite-N} defines
$G_N(t)=P_t^NG_N^0$, with marginals $G_{N,k}(t)$.

Unless otherwise specified, constants depend only on
$d,s,q,\sigma,T$, the domain, the potential, the bounds in
\Cref{ass:limiting-solution}, $A,\beta_2$, and the spatial exponents
$M,M_2,\alpha$ on $\R^d$. They are independent of $N,k$ and the
approximation parameters. We use the regularizations
$\phi_\varepsilon,K_\varepsilon$ of
\Cref{lem:common-riesz-regularization} for
$0<\varepsilon<\varepsilon_0$.

Let $e_{N,k}^{\varepsilon}$ be the energy
\eqref{eq:comparison-energy} with $\phi$ replaced by $\phi_\varepsilon$.
For $1\le k\le N$, we set
\begin{align}
 &L_{N,k}
 :=\sum_{i=1}^k v_i\cdot\nabla_{x_i}
 +\frac1N\sum_{i\ne j\le k}K(x_i-x_j)\cdot\nabla_{v_i},
 \label{eq:common-internal-generator}\\
 &J_{i,k}[h](Z_k)
 :=\int_{\mathcal D\times\R^d}
 K(x_i-y)h(Z_k,y,w)\,\dd y\dd w.
 \label{eq:common-flux-definition}
\end{align}
Write $L_{N,k}^{\varepsilon}$ and $J_{i,k}^{\varepsilon}$ for the
operators in \eqref{eq:common-internal-generator} and
\eqref{eq:common-flux-definition} with $K$ replaced by $K_\varepsilon$.

Let $G_N^\varepsilon(t)$ be the regularized particle law started from
$G_N^0$. Its marginals $G_{N,k}^\varepsilon$ satisfy the BBGKY hierarchy
\begin{equation*}
 \partial_tG_{N,k}^\varepsilon+L_{N,k}^\varepsilon G_{N,k}^\varepsilon
 -\sigma\Delta_{V_k}G_{N,k}^\varepsilon
 +\frac{N-k}{N}\sum_{i=1}^k\nabla_{v_i}\cdot
        J_{i,k}^\varepsilon[G_{N,k+1}^\varepsilon]=0,
 \qquad k<N,
\end{equation*}
with the last term absent at $k=N$. Appendix~\ref{app:fixed-N-passage}
justifies the singular hierarchy and the weighted estimates by
approximation at fixed $N$.

The energy satisfies
\begin{equation}\label{eq:common-energy-identities}
 L_{N,k}e_{N,k}=0,
 \qquad e_{N,k+1}\ge e_{N,k}+1+|v_{k+1}|^2.
\end{equation}
Indeed, we have
\begin{equation*}
 \begin{gathered}
 \nabla_{x_i}e_{N,k}=-\frac2N\sum_{j\ne i}K(x_i-x_j),\qquad
 \nabla_{v_i}e_{N,k}=2v_i,\\
 e_{N,k+1}-e_{N,k}=1+|v_{k+1}|^2+
             \frac2N\sum_{i=1}^k\phi(x_i-x_{k+1}).
 \end{gathered}
\end{equation*}
These identities also hold after regularization.

\subsection{Pointwise bounds and weights}\label{subsec:comparison-weights}

The following extension of \cite[Lemma~3.1]{JiangWuShortTime}
allows each field to depend on all particle positions.

\begin{lemma}\label{lem:common-exponential-envelope}
Let $1\le k\le N$, and let $b_i$ be velocity-independent fields
satisfying
\[
 b_i\in C([0,T];C_b^1(\mathcal D^k;\R^d)),
 \qquad \max_i\sup_{t\le T}\|b_i(t)\|_\infty\le L.
\]
Let $u_k^\varepsilon\in C([0,T];L^1)$ be the forward solution of
\[
 \partial_tu_k^\varepsilon+L_{N,k}^\varepsilon u_k^\varepsilon
 +\sum_i b_i(t,X_k)\cdot\nabla_{v_i}u_k^\varepsilon
 =\sigma\Delta_{V_k}u_k^\varepsilon,
\]
with a nonnegative integrable initial datum bounded, for some
$A_1,\beta_3>0$, by
\[
 u_k^\varepsilon(0,Z_k)\le A_1^k
 e^{-\beta_3e_{N,k}^\varepsilon}\prod_{i=1}^kw_m(x_i),
 \qquad
 w_m(x)=
 \begin{cases}
 \langle x\rangle^{-m},&\mathcal D=\R^d,\quad m>d,\\
 1,&\mathcal D=\T^d.
 \end{cases}
\]
There exist $\beta_T>0$ and $A_T\ge1$, depending only on
$d,\sigma,T,L,A_1,\beta_3$ and, on $\R^d$, $m$, such that
\[
 u_k^\varepsilon(t,Z_k)\le A_T^k
 e^{-\beta_Te_{N,k}^\varepsilon}\prod_{i=1}^kw_m(x_i),
 \qquad t\le T.
\]
These constants are independent of $N,k$ and $\varepsilon$.
The same estimate holds for any smooth bounded nonnegative even internal
potential with bounded first and second derivatives, using its
corresponding energy.
\end{lemma}

\begin{proof}
Put $m=0$ in the torus case and set
\[
 \overline u_k^\varepsilon=A(t)^k e^{-\beta(t)e_{N,k}^\varepsilon}
 \prod_iw_m(x_i),\qquad
 -\beta'=(4\sigma+2)\beta^2,\qquad
 \frac{A'}A=\frac{(m/2+2\beta L)^2}{8\beta^2},
\]
with $\beta(0)=\beta_3$ and $A(0)=A_1$.
For $\mathcal P_k^\varepsilon
=\partial_t+L_{N,k}^\varepsilon+\sum_i b_i\cdot\nabla_{v_i}
-\sigma\Delta_{V_k}$, equation \eqref{eq:common-energy-identities} implies
\begin{align*}
 \frac{\mathcal P_k^\varepsilon\overline u_k^\varepsilon}
 {\overline u_k^\varepsilon}
 k\frac{A'}A-\beta'e_{N,k}^\varepsilon
 -m\sum_i v_i\cdot\frac{x_i}{1+|x_i|^2}
 -2\beta\sum_i b_i(t,X_k)\cdot v_i+2d\sigma\beta k-4\sigma\beta^2\sum_i|v_i|^2,
\end{align*}
where the term containing $m$ is absent on $\T^d$.
Since $e_{N,k}^\varepsilon\ge k+|V_k|^2$ and
$|x|/(1+|x|^2)\le1/2$, Young's inequality yields
\begin{align*}
 \frac{\mathcal P_k^\varepsilon\overline u_k^\varepsilon}
 {\overline u_k^\varepsilon}
 &\ge k\left(\frac{A'}A+(4\sigma+2)\beta^2+2d\sigma\beta\right)
    +\sum_i\left(2\beta^2|v_i|^2-(m/2+2\beta L)|v_i|\right)\\
 &\ge k\left(\frac{A'}A+(4\sigma+2)\beta^2+2d\sigma\beta
        -\frac{(m/2+2\beta L)^2}{8\beta^2}\right)\ge0.
\end{align*}
For fixed $N,k,\varepsilon$, the regularized force is bounded, and the
Gaussian and spatial factors imply
$\overline u_k^\varepsilon\in C([0,T];L^1)$ and
$\mathcal P_k^\varepsilon\overline u_k^\varepsilon,
\nabla_{V_k}\overline u_k^\varepsilon\in L^1(0,T;L^1)$.
The comparison principle \eqref{eq:common-L1-order} therefore implies
$u_k^\varepsilon\le\overline u_k^\varepsilon$.
Solving the equation for $\beta$, we obtain
\[
 \beta(t)=\frac{\beta_3}
 {1+(4\sigma+2)\beta_3t}>0,
 \qquad
 \beta_T=\beta(T),\qquad A_T=\max\{1,A(T)\}.
\]
This completes the proof.
\end{proof}

At $k=1$, the energy is $1+|v|^2$. By
\Cref{lem:common-L1-stability} and
\eqref{eq:intro-velocity-derivative}, $f$ is the forward solution from
$f_0$ with prescribed field $E_f$. Applying
\Cref{lem:common-exponential-envelope} with $b_1=E_f$,
$\beta_3=a_0$ and $A_1=C_0e^{a_0}$ yields constants $a_f>0$ and
$C_f<\infty$ such that
\begin{equation}\label{eq:common-derived-f-envelope}
 f(t,x,v)\le C_fe^{-a_f|v|^2}
 \begin{cases}
  \langle x\rangle^{-M},&\mathcal D=\R^d,\\
  1,&\mathcal D=\T^d,
 \end{cases}
 \qquad 0\le t\le T.
\end{equation}
On $\R^d$, choose
\begin{equation}\label{eq:TV-alpha-choice}
 d<M_2<M,\qquad d(q-1)<\alpha<(q-1)M_2,
\end{equation}
and use the spatial weight
\[
 \varpi(x)=\langle x\rangle^\alpha,
 \qquad
 \varpi_k(X_k)=\prod_{i=1}^k\langle x_i\rangle^\alpha.
\]
On $\T^d$ set $\varpi\equiv1$ and $\varpi_k\equiv1$. On $\R^d$,
\eqref{eq:TV-alpha-choice} implies
\begin{equation}\label{eq:common-weight-compatibility}
 \langle x\rangle^{\alpha-(q-1)M_2}\le1,
 \qquad
 \int_{\R^d}\langle x\rangle^{\alpha-qM_2}\,\dd x<\infty,
 \qquad
 \int_{\R^d}\langle x\rangle^{-\alpha/(q-1)}\,\dd x<\infty.
\end{equation}
For a positive spatial weight $\varpi$, define
\begin{equation}\label{eq:common-kernel-norm}
 \|K\|_{q,\varpi}
 :=\left(\sup_{x\in\mathcal D}\int_{\mathcal D}|K(x-y)|^{q'}
          \varpi(y)^{-1/(q-1)}\,\dd y\right)^{1/q'}.
\end{equation}
By \Cref{lem:common-riesz-regularization}, the conditions
$(s+1)q'<d$ and, on $\R^d$, $\alpha>d(q-1)$ imply
\begin{equation}\label{eq:kernel-assumptions}
 \sup_{0<\varepsilon<\varepsilon_0}\|K_\varepsilon\|_{q,\varpi}<\infty,
 \qquad
 \|K_\varepsilon-K\|_{q,\varpi}\longrightarrow0,
 \qquad
 \|K_\varepsilon-K\|_{L^1(\mathcal D)}
 \le C\varepsilon^{d-s-1}.
\end{equation}
Equation \eqref{eq:common-derived-f-envelope} implies
$\sup_{t\le T}\|\rho_f(t)\|_\infty<\infty$, while
\Cref{ass:limiting-solution} ensures $\|\rho_f(t)\|_1=1$.
Consequently, uniformly in $t,\varepsilon$, we have
\[
 (|K_\varepsilon|*\rho_f)(t,x)
 \le C\|\rho_f(t)\|_\infty
          \int_0^{r_0}r^{d-s-2}\,\dd r
      +C\|\rho_f(t)\|_1\le C_T.
\]
The same estimate for $K$ yields
\begin{equation*}
 \sup_{t\le T}\left(
 \||K|*\rho_f(t)\|_\infty
 +\sup_{0<\varepsilon<\varepsilon_0}\||K_\varepsilon|*\rho_f(t)\|_\infty
 \right)<\infty.
\end{equation*}

Choose $C_1\ge1$ so that
$e_{N,k}^\varepsilon\le C_1e_{N,k}$, and fix
\[
 \beta_0=\frac{\beta_2}{4C_1},\qquad
 \beta_1=\frac{\beta_2}{2C_1}.
\]
The initial pointwise bound \eqref{eq:comparison-initial-envelope} then implies
\begin{equation}\label{eq:comparison-initial-regularized-envelope}
 G_{N,k}^0\le A^k e^{-\beta_1e_{N,k}^\varepsilon}
 \prod_{i=1}^k\langle x_i\rangle^{-M_2}.
\end{equation}
For the initial-data approximation in
\Cref{lem:common-smooth-initial-data}, use the exponent $\beta_0$ and set
\[
 \beta_T=\frac{\beta_0}{1+(4\sigma+2)\beta_0T},
 \qquad 0<\beta_T<\beta_0<\beta_1<\beta_2\le\beta.
\]

\subsection{Weighted estimate for the BBGKY interaction }\label{subsec:common-weighted-tools}

\begin{lemma}[{\cite[Lemma~3.2]{JiangWuShortTime}}]\label{lem:common-weighted-holder}
Let $1<q<\infty$, $\lambda>0$, and let $\varpi:\mathcal D\to(0,\infty)$
be measurable. Let $K$, $\mathcal E_k$ and
$\widehat{\mathcal E}_{k+1}$ be measurable and assume
\[
 \|K\|_{q,\varpi}<\infty,
 \qquad
 \widehat{\mathcal E}_{k+1}\ge\mathcal E_k+1+|v_{k+1}|^2.
\]
Set $\varpi_k=\prod_{j=1}^k\varpi(x_j)$, and use
\eqref{eq:common-kernel-norm} and \eqref{eq:common-flux-definition}.
For every measurable $h$ for which the right-hand side is finite,
$J_{i,k}[h]$ is absolutely convergent almost everywhere and
\begin{equation*}
 \int|J_{i,k}[h]|^qe^{\lambda\mathcal E_k}\varpi_k
 \le\Gamma_{d,q}(\lambda)\|K\|_{q,\varpi}^q
       \int|h|^qe^{\lambda\widehat{\mathcal E}_{k+1}}\varpi_{k+1},
\end{equation*}
where
\begin{equation*}
 \Gamma_{d,q}(\lambda)
 =e^{-\lambda}\left(\frac{\pi(q-1)}\lambda\right)^{d(q-1)/2}
 \le C_{d,q}\lambda^{-d(q-1)/2}.
\end{equation*}
\end{lemma}

\subsection{Comparison with the limiting equation}

For $1\le k\le N$, let $U_{N,k}$ be the forward density solving
\begin{equation}\label{eq:TV-U-PDE}
 \partial_tU_{N,k}+L_{N,k}U_{N,k}
 +\frac{N-k}{N}\sum_{i=1}^kE_f(t,x_i)\cdot\nabla_{v_i}U_{N,k}
 =\sigma\Delta_{V_k}U_{N,k},
 \qquad U_{N,k}(0)=G_{N,k}^0.
\end{equation}
Existence follows from
\Cref{prop:common-finite-N,prop:common-regularization-removal}.
Define $U_{N,k}^\varepsilon$ by replacing $L_{N,k}$ with
$L_{N,k}^\varepsilon$ in \eqref{eq:TV-U-PDE}, with the same prescribed
field $E_f$ and initial datum $G_{N,k}^0$.
The auxiliary equations are defined separately for each $k$.
The remainders $U_{N,k+j}-U_{N,k}f_t^{\otimes j}$, for $j=1,2$,
are estimated in \Cref{lem:ft-two-label-residual}.

We recall \cite[Lemmas~3.3--3.4]{JiangWuShortTime} with the initial
exponent $\beta_0$ and spatial exponent $M_2$ fixed in
Section~\ref{subsec:comparison-weights}. The initial bounds are
\eqref{eq:comparison-initial-envelope} and
\eqref{eq:comparison-initial-regularized-envelope}; the kernel bounds
are given by \eqref{eq:kernel-assumptions}. The regularity of the
prescribed field and the integrability of the velocity derivative
needed in the comparison follow from
\eqref{eq:intro-field-regularity}--\eqref{eq:intro-velocity-derivative}.
Thus the cited estimates apply on the prescribed interval $[0,T]$.

\begin{lemma}[{\cite[Lemma~3.3]{JiangWuShortTime}}]\label{lem:TV-U-pointwise}
For $t\in[0,T]$, the auxiliary solutions of \eqref{eq:TV-U-PDE} and its
regularized version satisfy
\begin{equation}\label{eq:TV-U-bound}
\begin{aligned}
 U_{N,k}(t)&\le A_T^ke^{-\beta_Te_{N,k}}
                     \prod_{j=1}^kw_{M_2}(x_j),
                     \\
 U_{N,k}^\varepsilon(t)&\le A_T^ke^{-\beta_Te_{N,k}^\varepsilon}
                     \prod_{j=1}^kw_{M_2}(x_j),
\end{aligned}
\end{equation}
where $w_{M_2}(x)=\langle x\rangle^{-M_2}$ on $\R^d$ and $w_{M_2}=1$
on $\T^d$. The constant $A_T\ge1$ is independent of
$N,k,\varepsilon$, and $\beta_T$ is fixed above. Moreover,
$U_{N,k}\in C([0,T];L^1)$ solves \eqref{eq:TV-U-PDE} in distributions
with initial datum $G_{N,k}^0$. At $k=N$, we have
\begin{equation}\label{eq:TV-level-N-equality}
 U_{N,N}=G_N,\qquad U_{N,N}^\varepsilon=G_N^\varepsilon.
\end{equation}
There is $B_T\ge1$ such that the particle marginals satisfy
\begin{align}
 G_{N,k}^\varepsilon(t)&\le B_T^Ne^{-\beta_Te_{N,k}^\varepsilon}
                         \prod_{j=1}^kw_{M_2}(x_j),
                         \label{eq:TV-GN-bound-regularized}\\
 G_{N,k}(t)&\le B_T^Ne^{-\beta_Te_{N,k}}
                         \prod_{j=1}^kw_{M_2}(x_j).
                         \label{eq:TV-GN-bound}
\end{align}
\end{lemma}

\begin{remark}\label{rem:auxiliary-pointwise-bounds}
For $0<\lambda\le(q-1)\beta_T/2$, the inequalities
$q\beta_T-\lambda\ge\beta_T$ and $\alpha-qM_2\le-M_2$ on $\R^d$,
together with \eqref{eq:common-weight-compatibility} and
\eqref{eq:TV-GN-bound}, give
\begin{align*}
 \int(1+|V_k|^2)G_{N,k}(t)^q e^{\lambda e_{N,k}}\varpi_k\,\dd Z_k
 &\le B_T^{qN}\int(1+|V_k|^2)e^{-\beta_T(k+|V_k|^2)}
                         \prod_{i=1}^kw_{M_2}(x_i)\,\dd Z_k\\
 &\le B_T^{qN}\left(1+\frac{dk}{2\beta_T}\right)
 \left[\left(\frac\pi{\beta_T}\right)^{d/2}
             \int_{\mathcal D}w_{M_2}(x)\,\dd x\right]^k
 \le C_T^N.
\end{align*}
The last step uses $k\le N$ and
$1+dk/(2\beta_T)\le(1+d/(2\beta_T))^k$.
By \eqref{eq:TV-GN-bound-regularized}, the same estimate holds for
$G_{N,k}^\varepsilon$ with $e_{N,k}^\varepsilon$, uniformly in
$\varepsilon$.
\end{remark}

\begin{lemma}[{\cite[Lemma~3.4]{JiangWuShortTime}}]\label{lem:TV-U-product}
There exists $C_T\ge1$, independent of $N,k,\varepsilon$, such that
\begin{equation}\label{eq:TV-U-product-regularized-sharp}
\begin{aligned}
 \sup_{t\le T}\|U_{N,k}(t)-f_t^{\otimes k}\|_1
 &\le \|G_{N,k}^0-f_0^{\otimes k}\|_1+\frac{C_Tk^{2-1/q}}N,
 \\
 \sup_{t\le T}\|U_{N,k}^\varepsilon(t)-f_t^{\otimes k}\|_1
 &\le \|G_{N,k}^0-f_0^{\otimes k}\|_1+\frac{C_T}{N}
       \bigl(k^{2-1/q}+k^2\varepsilon^{d-s-1}\bigr).
\end{aligned}
\end{equation}
By \eqref{eq:conditioned-initial-error}, both right-hand sides are
bounded by $C_T^k\varepsilon_N$ for $0<\varepsilon<\varepsilon_0$,
after increasing $C_T$.
\end{lemma}

\subsection{A weighted energy inequality}

For $0<\lambda_0\le1$, we set
\begin{equation}\label{eq:TV-global-lambda}
 \lambda(t)=\frac{\lambda_0}{1+\Lambda\lambda_0t},
 \qquad \lambda'=-\Lambda\lambda^2,
\end{equation}
where $\Lambda$ is given by \eqref{eq:common-flux-constants}.
We use the following form of \cite[Lemma~3.6]{JiangWuShortTime};
see also \cite[Lemma~9, Section~3.2]{BJS2025}.

\begin{lemma}\label{lem:common-weighted-flux-energy}
Let $2\le q<\infty$ and let $\lambda$ be given by
\eqref{eq:TV-global-lambda} with $0<\lambda_0\le1$.  Suppose that $H$ and
$\mathcal F_i$ are smooth and satisfy
\[
 \partial_tH+L_{N,k}^\varepsilon H-\sigma\Delta_{V_k}H
 +\sum_{i=1}^k\nabla_{v_i}\cdot\mathcal F_i=0,
\]
and set
\[
 W=e^{\lambda e_{N,k}^\varepsilon}\varpi_k,\qquad
 Y_k(t)=\int|H|^qW,\qquad
 \mathcal Q_k(t)=\sum_{i=1}^k\int|\mathcal F_i|^qW.
\]
Assume that $Y_k,\mathcal Q_k\in L^1(0,T)$ and
\[
 \int_0^T\!\int(1+|V_k|^2)|H|^qW\,\dd Z_k\dd t<\infty.
\]
Set
\begin{equation}\label{eq:common-flux-constants}
 \Lambda=4\left(1+4\sigma+\frac{q(q-1)}\sigma+q^2\right),\qquad
 C_2=1+2d\sigma+\frac{q(q-1)}\sigma+q^2.
\end{equation}
Then the following inequality holds in distributions on $(0,T)$:
\begin{equation*}
 Y_k'(t)\le
 C_2\bigl(1+\mathbf1_{\{\mathcal D=\R^d\}}\alpha^2\lambda(t)^{-2}\bigr)
 kY_k(t)
 +C_2\lambda(t)^{-(q-2)}
 \sum_{i=1}^k\int|\mathcal F_i|^qW.
\end{equation*}

The same estimate holds for every smooth bounded nonnegative even pair potential
with bounded first and second derivatives, using its generator and energy.
\end{lemma}

\begin{proof}
Using \eqref{eq:common-energy-identities} and
$\Delta_{V_k}e_{N,k}^\varepsilon=2dk$, we obtain $\nabla_{v_i}W=2\lambda v_iW$ and
\begin{align*}
 \frac{\partial_tW+L_{N,k}^\varepsilon W+\sigma\Delta_{V_k}W}{W}
 =\lambda'e_{N,k}^\varepsilon+2d\sigma\lambda k
       +4\sigma\lambda^2|V_k|^2
       +\ind_{\{\mathcal D=\R^d\}}
         \alpha\sum_i\frac{v_i\cdot x_i}{1+|x_i|^2}.
\end{align*}
Estimate the two flux terms in
\eqref{eq:fixed-N-localized-energy-identity} by Young's inequality,
as in \cite[proof of Lemma~3.6, estimates following (3.45)]{JiangWuShortTime}.
On $\R^d$, we also have
\[
 \frac{|L_{N,k}^\varepsilon\varpi_k|}{\varpi_k}
 \le\frac\alpha2\sum_i|v_i|
 \le\frac\Lambda4\lambda^2|V_k|^2
       +\frac{\alpha^2}{4\Lambda}\lambda^{-2}k.
\]
Insert these bounds, $\lambda'=-\Lambda\lambda^2$ and
$e_{N,k}^\varepsilon\ge k+|V_k|^2$ into
\eqref{eq:fixed-N-localized-energy-identity}.
The cutoff limit justified in Section~\ref{subsec:common-cutoff-testing}
yields
\begin{align*}
& Y_k'+\frac{3q(q-1)\sigma}{4}
       \sum_i\int |H|^{q-2}|\nabla_{v_i}H|^2W\\
 \le&\left[2d\sigma\lambda+
     \left(\frac{q(q-1)}\sigma+q^2-\Lambda\right)\lambda^2
     +\ind_{\{\mathcal D=\R^d\}}\frac{\alpha^2}{4\Lambda}\lambda^{-2}
     \right]kY_k\\
 &+\left(4\sigma+1-\frac{3\Lambda}{4}\right)
                  \lambda^2\int |V_k|^2|H|^qW
       +\left(\frac{q(q-1)}\sigma+q^2\right)
                  \lambda^{-(q-2)}\mathcal Q_k.
\end{align*}
With the choices \eqref{eq:common-flux-constants} and $\lambda\le1$, we deduce
\begin{align}\label{eq:common-flux-dissipation}
& Y_k'+\frac{3q(q-1)\sigma}{4}
     \sum_i\int |H|^{q-2}|\nabla_{v_i}H|^2W
     +\frac{\Lambda}{4}\lambda^2\int |V_k|^2|H|^qW\notag\\
 \le &C_2\bigl(1+\ind_{\{\mathcal D=\R^d\}}\alpha^2\lambda^{-2}\bigr)kY_k
       +C_2\lambda^{-(q-2)}\mathcal Q_k.
\end{align}
This completes the proof.
\end{proof}

\section{Propagation of chaos on a fixed time interval}
\label{sec:fixed-time-riesz}

We first compare the evolution from $G_N^0$ with $f_t^{\otimes k}$,
then remove the conditioning to prove \Cref{thm:fixed-time-riesz}.
Use the auxiliary solutions of Section~\ref{sec:common-comparison}
and the error scale \eqref{eq:comparison-error-scale}. Set
\[
 \kappa=q(d-s-1)-d>0.
\]
On $\R^d$, use the exponents and spatial weights of
\eqref{eq:TV-alpha-choice}. On $\T^d$, set $\alpha=0$ and
$\varpi_k=1$, and omit the spatial factors.

The constructions below have three distinct purposes. The splitting
radius $a$ controls the norm of the near-field force in the hierarchy
estimate. The map $\Theta_B$ modifies the empirical far-field force
to give a bound independent of $a$. For the energy calculations,
$h$ regularizes the singular force and $n$ indexes smooth
approximations of the conditioned initial density. We choose $a$
for the prescribed interval $[0,T]$ before taking these approximation
limits. The order of the choices is specified at the end of the next
subsection.

\subsection{Decomposition of the potential}

Write $c=c_{d,s}$. Choose $\chi\in C^\infty([0,\infty);[0,1])$
equal to one on $[0,1]$ and zero on $[2,\infty)$, and define
\begin{equation}\label{eq:ft-potential-split}
 \begin{aligned}
 \widetilde\phi^a(x)
 &=cs\int_{|x|}^\infty\chi(r/a)r^{-s-1}\,\dd r,
 &\widetilde K^a(x)&=cs\chi(|x|/a)\frac{x}{|x|^{s+2}},\\
 \overline\phi^a&=\phi-\widetilde\phi^a,
 &\overline K^a&=K-\widetilde K^a.
 \end{aligned}
\end{equation}
On $\T^d$, take $0<a\le a_2$ so that $B_{2a}$ lies in
the singular coordinate chart, and extend $\widetilde\phi^a$ by zero.
Fix $a_2\le1$ sufficiently small that
\begin{equation}\label{eq:ft-split-properties}
 0\le\widetilde\phi^a\le\phi,\qquad
 0\le\overline\phi^a\le C_a<\infty.
\end{equation}
For $0<|x|\le2a$ in the torus chart,
\eqref{eq:intro-periodic-riesz-potential} implies
\[
 \begin{aligned}
 \overline\phi^a(x)
 &=H_{\T,s}(x)+cs\int_{|x|}^{\infty}
                    (1-\chi(r/a))r^{-s-1}\,\dd r\\
 &\ge H_{\T,s}(x)+c(2a)^{-s}\ge0
 \end{aligned}
\]
after decreasing $a_2$. Outside $B_{2a}$, we have
$\overline\phi^a=\phi\ge1$. On $\R^d$, the same calculation applies with
$H_{\T,s}$ replaced by $1$ in
\eqref{eq:intro-whole-riesz-potential}.
The force $\overline K^a$ is smooth, bounded, Lipschitz and odd,
with $\overline K^a(0)=0$. On $\R^d$ it equals $K$ outside $B_{2a}$.
Set
\begin{equation}\label{eq:ft-far-force-size}
 M_a=\|\overline K^a\|_\infty\le C(1+a^{-s-1}),
 \qquad 0<a\le a_2.
\end{equation}
Indeed, the singular part
$(1-\chi(|x|/a))c_{d,s}s x/|x|^{s+2}$ vanishes on $B_a$
and is bounded by $Ca^{-s-1}$ elsewhere. The smooth part and the
force outside the chart are bounded.
Radial integration and \eqref{eq:intro-q-range} yield
\begin{equation}\label{eq:ft-small-near-force}
 \begin{aligned}
 \|\widetilde K^a\|_{q,\varpi}^q
 \le\|\widetilde K^a\|_{L^{q'}}^q
 \le C\left(\int_0^{2a}r^{d-1-(s+1)q'}\,\dd r\right)^{q-1}
 \le Ca^\kappa,\qquad
 \|\widetilde K^a\|_1\le Ca^{d-s-1}.
 \end{aligned}
\end{equation}

The coefficient of $a^\kappa$ in \eqref{eq:ft-small-near-force} may be taken as
\begin{equation}\label{eq:ft-near-kernel-constant}
 C_3=(c_{d,s}s)^q\left(
 \frac{|\mathbb S^{d-1}|\,2^{d-(s+1)q'}}{d-(s+1)q'}\right)^{q-1}.
\end{equation}

Set $m_f^a=\overline K^a*\rho_f$. Choose
\begin{equation*}
 B=2+\|E_f\|_{C_tL^\infty}
       +C_{d,s}(a_2)^{d-s-1}
                        \|\rho_f\|_{L_t^\infty L^\infty},
 \qquad \Theta_B(u)=\chi(|u|/B)u.
\end{equation*}
Here $C_{d,s}$ is the constant in the $L^1$ estimate
\eqref{eq:ft-small-near-force}. Since
$m_f^a=E_f-\widetilde K^a*\rho_f$, we obtain
\begin{equation}\label{eq:ft-clipping-margin}
 \sup_{0<a\le a_2}\sup_{t\le T}\|m_f^a(t)\|_\infty
 \le B-1,\qquad B_0:=\|\Theta_B\|_\infty\le2B.
\end{equation}
The map $\Theta_B\in C_c^\infty(\R^d;\R^d)$ equals the identity on
$\{|u|\le B\}$ and has a Lipschitz constant bounded in terms of $\chi$.
Let $\widetilde G_N$ be the law, with initial density $G_N^0$, of the
modified system
\begin{equation}\label{eq:ft-clipped-system}
 \begin{aligned}
 \dd X_i&=V_i\,\dd t,\\
 \dd V_i&=\left\{\frac1N\sum_{j\ne i}\widetilde K^a(X_i-X_j)
   +\Theta_B\left(\frac1N\sum_{j\ne i}\overline K^a(X_i-X_j)\right)
          \right\}\dd t+\sqrt{2\sigma}\,\dd B_i.
 \end{aligned}
\end{equation}
Denote its trajectories by $(\widetilde X_i,\widetilde V_i)_{i=1}^N$.
Apply \Cref{lem:ft-approximation}\ref{item:ft-approximation-dynamics}
with the same cutoff $\chi$ and
\[
 b_i^N(X_N)=\Theta_B\left(\frac1N\sum_{j\ne i}
                                         \overline K^a(x_i-x_j)\right).
\]
This field is smooth and permutation-equivariant, with
$\sup_{N,i}\|b_i^N\|_\infty\le B_0$. We use the approximations
$G_N^{0,h,n},\widetilde G_N^{h,n},U_{N,k}^{h,n}$ constructed in
Section~\ref{subsec:ft-approximation}, with
$0<h<a/10$ and $\xi_{n,h}=\|G_N^{0,h,n}-G_N^0\|_1$.
Here the first density approximates $G_N^0$, and the other two are
the corresponding modified and auxiliary evolutions.
The auxiliary equation uses the full regularized kernel
$K^h=\widetilde K^{a,h}+\overline K^a$ and the prescribed field
$(N-k)E_f/N$ from \eqref{eq:TV-U-PDE}.

With the exponents $\beta_0,\beta_1$ fixed in
Section~\ref{sec:common-comparison}, set
\begin{equation*}
 \begin{gathered}
 \beta_4=\tfrac12\min\{\beta_T,a_f\},\qquad
 \lambda_0=\tfrac18\min\{1,(q-1)\beta_4,q\beta_T\}.
 \end{gathered}
\end{equation*}
Here $\beta_T$ is from \eqref{eq:TV-U-bound} and $a_f$ is from
\eqref{eq:common-derived-f-envelope}.
Use $\lambda$ from \eqref{eq:TV-global-lambda} with this $\lambda_0$
and $\Lambda,C_2$ from \eqref{eq:common-flux-constants}. Then
\begin{equation}\label{eq:ft-exponent-margin}
 0<4\lambda_0<\min\{1,(q-1)\beta_4,q\beta_T\}.
\end{equation}
We define
\begin{equation}\label{eq:ft-weights}
 \begin{aligned}
 & \widetilde e_{N,k}
=\sum_{i=1}^k(1+|v_i|^2)
        +\frac1N\sum_{i\ne j\le k}\widetilde\phi^a(x_i-x_j),\\
 &\widetilde W_{N,k}(t)=e^{\lambda(t)\widetilde e_{N,k}}\varpi_k, \quad 
 W_{N,k}(t)=e^{\lambda(t)e_{N,k}}\varpi_k.
 \end{aligned}
\end{equation}
Superscript $h$ on either weight replaces its energy by
$\widetilde e_{N,k}^h$ or $e_{N,k}^h$, respectively. Since the far-field potential
is unchanged by regularization, we have
\begin{equation}\label{eq:ft-full-near-energy}
 \begin{aligned}
 0\le e_{N,k}^h-\widetilde e_{N,k}^h
 =\frac1N\sum_{i\ne j\le k}\overline\phi^{a}(x_i-x_j)\le\|\overline\phi^{a}\|_\infty\frac{k(k-1)}N
 \le\|\overline\phi^{a}\|_\infty k.
 \end{aligned}
\end{equation}
The same identity and bounds hold without the superscript $h$.

The parameters are chosen in the following order. First fix the
spatial exponents and $\beta_0,\beta_1,\beta_T$ as in
Section~\ref{sec:common-comparison}. Then choose $B,\beta_4,\lambda_0$
as above. These constants are independent of $a,N,h,n$.
By \Cref{rem:smoothing-general-potential}, the constant $A_2$ in
\eqref{eq:ft-smooth-initial-envelope} is also independent of $a,N,h,n$.
Set
\begin{equation*}
 R_0=A_2^{q-1}.
\end{equation*}
The constants $C_5,C_6$ in \eqref{eq:ft-near-coefficients} then
determine $a_T$ and $R_T$ through \eqref{eq:ft-explicit-radius}.
The estimates below hold for $0<a\le a_T$; we take $a=a_T$ in the
final comparison. At fixed $N,a$, choose the approximation parameters
$h,n$ with $0<h<a/10$, using the smooth initial data indexed by $n$.
In the energy calculation, we first remove the phase-space cutoff
and then the smoothing of the prescribed field $E_f$. We next let
$n\to\infty$ and $h\downarrow0$, in this order. Constants involving
derivatives of $\overline K^{a_T}$ may depend on $a_T$; the constants
in the marginal estimates are independent of $N,k,h,n$.

\Cref{lem:ft-approximation} supplies the energy tests and convergence
used below. Its pointwise bounds at fixed $N$ follow from
\Cref{lem:common-exponential-envelope} and the construction of the
dynamics in Section~\ref{sec:kernel-verification}. These bounds
justify localization and passage to the limit. The estimates uniform
in $N$ are established in this section from the regularized equations.

\subsection{Weighted estimates for the modified system}

We define
\begin{equation}\label{eq:ft-near-coefficients}
 \begin{aligned}
 a_B(t)&=C_2\left(1+\ind_{\{\mathcal D=\R^d\}}\alpha^2\lambda(t)^{-2}
                  +2^{q-1}B_0^q\lambda(t)^{-(q-2)}\right),\\
 \ell_a(t)&=2^{q-1}C_2 C_3a^\kappa
                      \lambda(t)^{-(q-2)}\Gamma_{d,q}(\lambda(t)),\\
 C_5&=\max_{t\le T}a_B(t),\qquad
 C_6=2^{q-1}C_2 C_3\lambda(T)^{-(q-2)}
                                      \Gamma_{d,q}(\lambda(T)).
 \end{aligned}
\end{equation}
Since $\lambda$ decreases and $q\ge2$, we have
\begin{equation}\label{eq:ft-cap-coefficients}
 a_B(t)\le C_5,\qquad \ell_a(t)\le C_6a^\kappa,
 \qquad t\le T,
\end{equation}
and $C_5,C_6$ are independent of $a$. We set
\begin{equation}\label{eq:ft-explicit-radius}
 a_T=\min\left\{a_2,
       \left(2R_0TC_6e^{TC_5}\right)^{-1/\kappa}\right\},
 \qquad R_T=2R_0e^{TC_5}.
\end{equation}

\begin{lemma}[Weighted bounds on the prescribed interval]
\label{lem:ft-global-cap}
For every $0<a\le a_T$ and $1\le k\le N$, we have
\begin{equation}
 \sup_{t\le T}\int\widetilde G_{N,k}(t)^q\widetilde W_{N,k}(t)
 \le R_T^k.\label{eq:ft-near-cap}
\end{equation}
The constant $R_T$ is independent of $a,N,k$.
The estimate holds for $\widetilde G_{N,k}^{h,n}$ with the
corresponding regularized weight, uniformly in $h,n$.
\end{lemma}

\begin{proof}
Work first at fixed $N,a,h,n$ and put
\[
 \mathcal A_{N,k}^{h,n}(t)
   =\int(\widetilde G_{N,k}^{h,n}(t))^q\widetilde W_{N,k}^h(t),\qquad
 I_{i,k}^{h,n}=\int b_i^N\widetilde G_N^{h,n}\,\dd Z_{k+1:N}.
\]
Positivity gives
\begin{equation}\label{eq:ft-positive-bounded-flux}
 |I_{i,k}^{h,n}|\le B_0\widetilde G_{N,k}^{h,n}.
\end{equation}
Let $\widetilde J_{i,k}^h$ be defined by
\eqref{eq:common-flux-definition} with $K$ replaced by
$\widetilde K^{a,h}$. The flux in the marginal equation with internal
force $\widetilde K^{a,h}$ is
$((N-k)/N)\widetilde J_{i,k}^h[\widetilde G_{N,k+1}^{h,n}]
 +I_{i,k}^{h,n}$. Since
$|\widetilde K^{a,h}|\le|\widetilde K^a|$,
\Cref{lem:common-weighted-holder} yields
\[
 \|\widetilde J_{i,k}^h[\widetilde G_{N,k+1}^{h,n}]
                      \|_{L^q(\widetilde W_{N,k}^h)}^q
 \le C_3a^\kappa\Gamma_{d,q}(\lambda)
                              \mathcal A_{N,k+1}^{h,n}.
\]
Apply \Cref{lem:common-weighted-flux-energy} to the regularized
marginal equation. The pointwise bounds and localization argument in
\Cref{lem:ft-approximation}\ref{item:ft-approximation-limits}
justify this test at fixed $N,a,h,n$.
By \eqref{eq:ft-positive-bounded-flux},
$|u+v|^q\le2^{q-1}(|u|^q+|v|^q)$ and
\eqref{eq:ft-near-coefficients}, we obtain
\begin{equation}\label{eq:ft-smooth-near-energy}
 \begin{aligned}
 (\mathcal A_{N,k}^{h,n})'
 &\le ka_B\mathcal A_{N,k}^{h,n}
     +k\ell_a\left(\frac{N-k}{N}\right)^q
                         \mathcal A_{N,k+1}^{h,n}\\
 &\le ka_B\mathcal A_{N,k}^{h,n}
     +k\ell_a\frac{N-k}{N}\mathcal A_{N,k+1}^{h,n}.
 \end{aligned}
\end{equation}
The last inequality uses $((N-k)/N)^q\le(N-k)/N$.
The terms involving level $k+1$ are
absent at $k=N$. The integrated inequality holds on every
$[s,t]\subset[0,T]$.
Since $\lambda_0<(q-1)\beta_4<(q-1)\beta_0$,
\eqref{eq:ft-smooth-initial-envelope} implies
\[
 \begin{aligned}
 (G_{N,k}^{0,h,n})^{q-1}\widetilde W_{N,k}^h(0)
 &\le A_2^{(q-1)k}
 e^{-((q-1)\beta_0-\lambda_0)\widetilde e_{N,k}^h}
 \prod_{i=1}^k\langle x_i\rangle^{\alpha-(q-1)M_2}\\
 &\le A_2^{(q-1)k},\qquad
 \mathcal A_{N,k}^{h,n}(0)\le R_0^k.
 \end{aligned}
\]

We define
\begin{equation*}
 R(t)=\frac{R_0\exp\!\left(\int_0^t a_B(u)\,\dd u\right)}
 {1-R_0\int_0^t\ell_a(u)
                   \exp\!\left(\int_0^u a_B(v)\,\dd v\right)\,\dd u}.
\end{equation*}
For $0<a\le a_T$, it follows from \eqref{eq:ft-explicit-radius} that
\begin{equation*}
 R_0\int_0^T\ell_a(u)
                   \exp\!\left(\int_0^u a_B(v)\,\dd v\right)\,\dd u
 \le R_0TC_6a^\kappa e^{TC_5}
 \le R_0TC_6a_T^\kappa e^{TC_5}\le\tfrac12.
\end{equation*}
Hence $R'=a_BR+\ell_aR^2$ and $R\le R_T$ on $[0,T]$.
At the terminal level, we deduce from \eqref{eq:ft-smooth-near-energy} that
\begin{equation*}
 \mathcal A_{N,N}^{h,n}(t)
 \le R_0^N\exp\!\left(N\int_0^t a_B(u)\,\dd u\right)
 \le R(t)^N.
\end{equation*}
If $\mathcal A_{N,k+1}^{h,n}\le R^{k+1}$, then
\eqref{eq:ft-smooth-near-energy} and
$(R^k)'=ka_BR^k+k\ell_aR^{k+1}$ imply
\[
 \frac{\dd}{\dd t}\left[
 e^{-k\int_0^t a_B(u)\,\dd u}
       (\mathcal A_{N,k}^{h,n}-R^k)\right]
 \le-\frac{k^2}{N}\ell_aR^{k+1}
               e^{-k\int_0^t a_B(u)\,\dd u}\le0.
\]
Integration yields
\begin{equation*}
 \mathcal A_{N,k}^{h,n}(t)-R(t)^k
 \le\exp\!\left(k\int_0^t a_B(u)\,\dd u\right)
       \bigl(\mathcal A_{N,k}^{h,n}(0)-R_0^k\bigr)\le0.
\end{equation*}
Downward induction proves \eqref{eq:ft-near-cap} uniformly in $h,n$.

At fixed $N,a$, apply
\Cref{lem:ft-approximation}\ref{item:ft-approximation-weights}--\ref{item:ft-approximation-limits}
to the preceding inequalities. The weighted integrals converge
uniformly on $[0,T]$, so \eqref{eq:ft-near-cap} and the integrated
form of \eqref{eq:ft-smooth-near-energy} pass to the limits
$n\to\infty$ and then $h\downarrow0$, on every
$[s,t]\subset[0,T]$.
\end{proof}

For the rest of this section, let $0<a\le a_T$ and choose a common
constant $A_T$ in \eqref{eq:TV-U-bound} and
\eqref{eq:ft-approximation-envelopes}, independent of $a,N,k,h,n$.
Using the unit masses, \eqref{eq:ft-exponent-margin} and
\eqref{eq:ft-full-near-energy}, we obtain
\begin{equation}\label{eq:ft-U-cap}
 \begin{aligned}
 \sup_{t\le T}\int U_{N,k}(t)^q\widetilde W_{N,k}(t)
 &\le\sup_{t\le T}\int U_{N,k}(t)^qW_{N,k}(t)
 \le A_T^{(q-1)k},\\
 \sup_{t\le T}\int(U_{N,k}^{h,n}(t))^q\widetilde W_{N,k}^h(t)
 &\le\sup_{t\le T}\int(U_{N,k}^{h,n}(t))^qW_{N,k}^h(t)
 \le A_T^{(q-1)k}.
 \end{aligned}
\end{equation}
In the remainder and error estimates below, dependence on
$\overline K^a$ is displayed through $M_a$; the remaining constants
are independent of $a,N,k,h,n$.

\subsection{Estimates for the auxiliary remainders}

For $j=1,2$ and $1\le k\le N-j$, define
\begin{equation}\label{eq:ft-remainder-definition}
 h_{N,k,j}=U_{N,k+j}-U_{N,k}f_t^{\otimes j}.
\end{equation}
Write $h_{N,k,j}^{h,n}$ for the corresponding regularized remainder.
We apply the interpolation argument of
\cite[Proposition~3.5]{JiangWuShortTime} at both values of $j$.

\begin{lemma}
\label{lem:ft-two-label-residual}
For $j=1,2$ and $1\le k\le N-j$, there is $D_T\ge1$, independent of
$a,N,k,h,n$, such that
\begin{equation*}
 \begin{aligned}
 \sup_{t\le T}\int|h_{N,k,j}|^q
                       e^{\lambda(t)\left(e_{N,k}+\sum_{\ell=k+1}^{k+j}(1+|v_\ell|^2)\right)}\varpi_{k+j}
       &\le D_T^k\varepsilon_N,\\
 \sup_{t\le T}\int|h_{N,k,j}^{h,n}|^q
                       e^{\lambda(t)\left(e_{N,k}^h+\sum_{\ell=k+1}^{k+j}(1+|v_\ell|^2)\right)}\varpi_{k+j}
       &\le D_T^k(\varepsilon_N+\xi_{n,h}).
 \end{aligned}
\end{equation*}
\end{lemma}

\begin{proof}
By \Cref{lem:TV-U-product}, we have
\[
 \|h_{N,k,j}\|_1
 \le \|U_{N,k+j}-f_t^{\otimes(k+j)}\|_1
       +\|U_{N,k}-f_t^{\otimes k}\|_1
 \le C_T^{k+j}\varepsilon_N.
\]
For the regularized equations,
\eqref{eq:ft-core-kernel-difference}--\eqref{eq:ft-core-regularization-bounds}
give the kernel bounds used in
\cite[proof of Lemma~3.4]{JiangWuShortTime}, uniformly in $a,h$.
Thus \eqref{eq:TV-U-product-regularized-sharp} holds with $K^h$
in place of $K_\varepsilon$. Markov contraction,
\eqref{eq:conditioned-initial-error} and $h<1$ yield
\begin{equation}\label{eq:ft-regularized-U-product}
 \sup_{t\le T}\|U_{N,m}^{h,n}-f_t^{\otimes m}\|_1
 \le C_T^m(\varepsilon_N+\xi_{n,h}).
\end{equation}
The preceding triangle inequality then gives
\[
 \|h_{N,k,j}^{h,n}\|_1\le C_T^{k+j}(\varepsilon_N+\xi_{n,h}).
\]
By \eqref{eq:TV-U-bound}, \eqref{eq:ft-approximation-envelopes}
and \eqref{eq:common-derived-f-envelope}, we obtain
\[
 |h_{N,k,j}|
 \le U_{N,k+j}+U_{N,k}f_t^{\otimes j}
 \le C_T^{k+j}e^{-\beta_4\left(e_{N,k}+\sum_{\ell=k+1}^{k+j}(1+|v_\ell|^2)\right)}
\prod_{\ell=1}^{k+j}\langle x_\ell\rangle^{-M_2}.
\]
The same bound holds for $h_{N,k,j}^{h,n}$ with energy $e_{N,k}^h$.
By \eqref{eq:ft-exponent-margin} and
\eqref{eq:common-weight-compatibility}, we have
\begin{align*}
 \int |h_{N,k,j}|^q
 e^{\lambda\left(e_{N,k}+\sum_{\ell=k+1}^{k+j}(1+|v_\ell|^2)\right)}\varpi_{k+j}
 &\le C_T^{(q-1)(k+j)}\|h_{N,k,j}\|_1
 \le D_T^k\varepsilon_N.
\end{align*}
The same interpolation with $e_{N,k}^h$ and $h_{N,k,j}^{h,n}$,
using \eqref{eq:ft-regularized-U-product}, proves the second estimate.
\end{proof}

For a bounded kernel $\Psi$ on $k+j$ variables, write
$T_\Psi h(Z_k)=\int\Psi(Z_{k+j})h(Z_{k+j})\,\dd Z_{k+1:k+j}$.
Applying the weighted H\"older inequality in $Z_{k+1:k+j}$,
integrating in $Z_k$ and using $\lambda(t)\ge\lambda(T)$, we obtain
\begin{equation}\label{eq:ft-bounded-projection}
 \|T_\Psi h\|_{L^q(\widetilde W_{N,k})}^q
 \le C_{T,j}\|\Psi\|_\infty^q
             \int|h|^qe^{\lambda(t)\left(\widetilde e_{N,k}+\sum_{\ell=k+1}^{k+j}(1+|v_\ell|^2)\right)}\varpi_{k+j},
 \qquad j=1,2,
\end{equation}
where
\[
 C_{T,j}=\left(\int_{\Omega_{\mathcal D}}
 e^{-\lambda(T)(1+|v|^2)/(q-1)}\varpi(x)^{-1/(q-1)}\,\dd x\dd v
             \right)^{j(q-1)}<\infty.
\]
Finiteness follows from \eqref{eq:common-weight-compatibility}.
Positivity of the potential also implies
\[
 \widetilde e_{N,k+j}\ge \widetilde e_{N,k}+\sum_{\ell=k+1}^{k+j}(1+|v_\ell|^2),\qquad
 \int|h|^qe^{\lambda \left(\widetilde e_{N,k}+\sum_{\ell=k+1}^{k+j}(1+|v_\ell|^2)\right)}\varpi_{k+j}
 \le\|h\|_{L^q(\widetilde W_{N,k+j})}^q.
\]
The same estimates hold for $\widetilde e^h,\widetilde W^h$.
Since $\widetilde e_{N,k}\le e_{N,k}$ and
$\widetilde e_{N,k}^h\le e_{N,k}^h$, the bounds in
\Cref{lem:ft-two-label-residual} also hold with the corresponding
near-field energies on their left-hand sides.

\subsection{A centered second-moment estimate}

We estimate the flux produced by the bounded modification through
a centered second moment. Define
\begin{equation*}
 \mathcal E_i^a(X_N)=\frac1N\sum_{j\ne i}\overline K^a(x_i-x_j),\qquad
 D_i=\Theta_B(\mathcal E_i^a)-\mathcal E_i^a,\qquad
 J_{i,k}^{D}=\int D_i\widetilde G_N\,\dd Z_{k+1:N}.
\end{equation*}
At $k=N$, $J_{i,N}^D=D_i\widetilde G_N$. We set
\begin{equation}\label{eq:ft-error-definition}
 H_{N,k}=\widetilde G_{N,k}-U_{N,k},\qquad
 Y_{N,k}(t)=\int|H_{N,k}(t)|^q\widetilde W_{N,k}(t).
\end{equation}
By \eqref{eq:TV-level-N-equality},
$H_{N,N}=\widetilde G_N-G_N$.
Use the same definitions for $H_{N,k}^{h,n},Y_{N,k}^{h,n}$ and
$J_{i,k}^{D,h,n}$, using $\widetilde W_{N,k}^h$ in $Y_{N,k}^{h,n}$.
From \eqref{eq:ft-near-cap} and \eqref{eq:ft-U-cap}, we obtain
\begin{equation}\label{eq:ft-error-cap}
 \begin{gathered}
 0\le Y_{N,k}(t),Y_{N,k}^{h,n}(t)
 \le2^{q-1}\bigl(R_T^k+A_T^{(q-1)k}\bigr)\le C_4^k,
 \qquad t\le T,\\
 C_4=2^q\max\{1,R_T,A_T^{q-1}\}.
 \end{gathered}
\end{equation}
The constant $C_4$ is independent of $a,N,k,h,n$.

\begin{lemma}
\label{lem:ft-clipping-flux}
There are $C_T,D_T\ge1$, independent of $a,N,k,h,n$, such that
\begin{equation}\label{eq:ft-clipping-flux-bound}
 \begin{aligned}
 \sum_{i=1}^k\|J_{i,k}^D\|_{L^q(\widetilde W_{N,k})}^q
 &\le C_T(1+M_a)^{2q}
                \bigl(kY_{N,k+2}+D_T^k\varepsilon_N\bigr),\\
 \sum_{i=1}^k\|J_{i,k}^{D,h,n}\|_{L^q(\widetilde W_{N,k}^h)}^q
 &\le C_T(1+M_a)^{2q}
       \bigl(kY_{N,k+2}^{h,n}+D_T^k(\varepsilon_N+\xi_{n,h})\bigr).
 \end{aligned}
\end{equation}
These estimates hold for $1\le k\le N$, with
$Y_{N,N+1}=Y_{N,N+2}=0$ and the same convention for the regularized
errors.
\end{lemma}

\begin{proof}
By \eqref{eq:ft-clipping-margin}, $\Theta_B(m_f^a)=m_f^a$.
On $\{D_i\ne0\}$, we have
\[
 |\mathcal E_i^a-m_f^a(t,x_i)|
 \ge|\mathcal E_i^a|-|m_f^a(t,x_i)|
 >B-(B-1)=1.
\]
Hence
\begin{equation*}
 \begin{aligned}
 |D_i|
 &=|\Theta_B(\mathcal E_i^a)-\Theta_B(m_f^a(t,x_i))
                       -(\mathcal E_i^a-m_f^a(t,x_i))|\\
 &\le(1+\operatorname{Lip}\Theta_B)
                         |\mathcal E_i^a-m_f^a(t,x_i)|
                         \ind_{\{D_i\ne0\}}\\
 &\le(1+\operatorname{Lip}\Theta_B)
                         |\mathcal E_i^a-m_f^a(t,x_i)|^2.
 \end{aligned}
\end{equation*}
Positivity gives
\begin{equation}\label{eq:ft-variance-domination}
 |J_{i,k}^D|\le(1+\operatorname{Lip}\Theta_B)V_{i,k},\qquad
 V_{i,k}=\int|\mathcal E_i^a-m_f^a(t,x_i)|^2
                       \widetilde G_N\,\dd Z_{k+1:N}.
\end{equation}
Fix $t$, $Z_k$ and $i\le k$, put $m=N-k$, and write
\[
 \xi_i(z)=\overline K^a(x_i-x)-m_f^a(t,x_i),\qquad
 A_{i,k}=\sum_{\substack{j\le k\\j\ne i}}\overline K^a(x_i-x_j)
                                                     -km_f^a(t,x_i).
\]
Since $\rho_f$ has unit mass, $|m_f^a|\le M_a$,
$|\xi_i|\le2M_a$ and $|A_{i,k}|\le2kM_a$. The centering identities are
\begin{equation}\label{eq:ft-centering}
 \int\xi_i(z)f_t(z)\,\dd z=0,\qquad
 N(\mathcal E_i^a-m_f^a(t,x_i))=A_{i,k}+\sum_{j>k}\xi_i(z_j).
\end{equation}
Expanding the square and using exchangeability, we obtain
\begin{align}\label{eq:ft-exact-variance}
 V_{i,k}
 &=\frac{|A_{i,k}|^2}{N^2}\widetilde G_{N,k}
   +\frac{2m}{N^2}A_{i,k}\cdot T_{\xi_i}\widetilde G_{N,k+1}\notag\\
 &\quad+\frac m{N^2}T_{|\xi_i|^2}\widetilde G_{N,k+1}
   +\frac{m(m-1)}{N^2}T_{\xi_i\cdot\xi_i}\widetilde G_{N,k+2}.
\end{align}
Here the last kernel is
$\xi_i(z_{k+1})\cdot\xi_i(z_{k+2})$; its coefficient counts ordered
pairs of distinct indices greater than $k$. Terms with an index
above $N$ are omitted.
For $k<N$, positivity and marginal consistency imply
\[
 |T_{\xi_i}\widetilde G_{N,k+1}|\le2M_a\widetilde G_{N,k},\qquad
 0\le T_{|\xi_i|^2}\widetilde G_{N,k+1}
                            \le4M_a^2\widetilde G_{N,k}.
\]
By \eqref{eq:ft-near-cap}, the $q$-th powers of the
$L^q(\widetilde W_{N,k})$ norms of the first three terms in
\eqref{eq:ft-exact-variance} are bounded, respectively, by
\[
 C_qM_a^{2q}R_T^k(k/N)^{2q},\qquad
 C_qM_a^{2q}R_T^k(mk/N^2)^q,\qquad
 C_qM_a^{2q}R_T^k(m/N^2)^q.
\]
For $k\le N-2$, it follows that
\begin{align}\label{eq:ft-small-variance-terms}
 \left\|V_{i,k}-\frac{m(m-1)}{N^2}
 T_{\xi_i\cdot\xi_i}\widetilde G_{N,k+2}
                         \right\|_{L^q(\widetilde W_{N,k})}^q\le C_qM_a^{2q}R_T^k
\left[(k/N)^{2q}+(k/N)^q+N^{-q}\right].
\end{align}
The polynomial factors satisfy
\begin{equation}\label{eq:ft-variance-polynomial-bound}
 kC_qR_T^k\left[(k/N)^{2q}+(k/N)^q+N^{-q}\right]
 \le\frac{C_qR_T^k(k^{2q+1}+k^{q+1}+k)}N
 \le\frac{D_T^k}N,
\end{equation}
where $D_T$ is independent of $a$.
For $k+2\le N$, we decompose
\[
 \widetilde G_{N,k+2}=H_{N,k+2}+h_{N,k,2}+U_{N,k}f_t^{\otimes2}.
\]
By the centering identity \eqref{eq:ft-centering}, we have
\[
 T_{\xi_i\cdot\xi_i}(U_{N,k}f_t^{\otimes2})
          =U_{N,k}\left|\int\xi_i f_t\right|^2=0.
\]
By \eqref{eq:ft-bounded-projection},
$\|\xi_i(z_{k+1})\cdot\xi_i(z_{k+2})\|_\infty\le4M_a^2$,
and \Cref{lem:ft-two-label-residual}, we obtain
\[
 \begin{aligned}
 \|T_{\xi_i\cdot\xi_i}\widetilde G_{N,k+2}\|_{L^q(\widetilde W_{N,k})}^q
 &=\|T_{\xi_i\cdot\xi_i}(H_{N,k+2}+h_{N,k,2})
                                      \|_{L^q(\widetilde W_{N,k})}^q\\
 &\le C_TM_a^{2q}\bigl(Y_{N,k+2}+D_T^k\varepsilon_N\bigr).
 \end{aligned}
\]
Sum over $i$ and apply
\eqref{eq:ft-variance-domination},
\eqref{eq:ft-small-variance-terms} and
\eqref{eq:ft-variance-polynomial-bound}.
Increasing the exponential base $D_T$ to absorb $kD_T^k$ proves
\eqref{eq:ft-clipping-flux-bound} for $k\le N-2$.
If $k=N-1$ or $k=N$, the last term of
\eqref{eq:ft-exact-variance} is absent. Equations
\eqref{eq:ft-variance-polynomial-bound} and
$N^{-1}\le\varepsilon_N$ give
\begin{equation*}
 \sum_{i=1}^k\|V_{i,k}\|_{L^q(\widetilde W_{N,k})}^q
 \le\frac{M_a^{2q}D_T^k}N
 \le(1+M_a)^{2q}D_T^k\varepsilon_N.
\end{equation*}
This proves the terminal cases.

The identities \eqref{eq:ft-centering}--\eqref{eq:ft-exact-variance}
also hold under $\widetilde G_N^{h,n}$, since $\overline K^a$,
$m_f^a$ and $D_i$ are unchanged. Use the regularized bounds in
\Cref{lem:ft-global-cap,lem:ft-two-label-residual} in the preceding
estimates to obtain \eqref{eq:ft-clipping-flux-bound} with
$\widetilde W_{N,k}^h$ and $\varepsilon_N+\xi_{n,h}$.
All constants apart from the displayed powers of $M_a$ are
independent of $a$.
\end{proof}

\subsection{The error hierarchy}

\begin{lemma}
\label{lem:ft-error-hierarchy}
There are $C_T,D_T\ge1$, independent of $a,N,k$, such that
\begin{equation}\label{eq:ft-error-hierarchy}
 \begin{aligned}
 Y_{N,k}(t)-Y_{N,k}(s)
 \le C_T\int_s^t\Bigl[&k\bigl((1+M_a)^qY_{N,k}
             +Y_{N,k+1}+(1+M_a)^{2q}Y_{N,k+2}\bigr)\\
             &+(1+M_a)^{2q}D_T^k\varepsilon_N\Bigr](u)\,\dd u,
 \qquad Y_{N,k}(0)=0,
 \end{aligned}
\end{equation}
for all $0\le s\le t\le T$ and $1\le k\le N$, with
$Y_{N,N+1}=Y_{N,N+2}=0$.
\end{lemma}

\begin{proof}
Define the internal near-field operator and the internal far-field force by
\begin{equation}\label{eq:ft-internal-far-force}
 \begin{aligned}
 \widetilde L_{N,k}^a
 &=\sum_{i=1}^kv_i\cdot\nabla_{x_i}
 +\frac1N\sum_{i\ne j\le k}\widetilde K^a(x_i-x_j)\cdot\nabla_{v_i},\\
 F_{i,k}^{a}
 &=\frac1N\sum_{\substack{j\le k\\j\ne i}}
                         \overline K^a(x_i-x_j),\qquad
 \|F_{i,k}^{a}\|_\infty\le\frac{k-1}{N}M_a\le M_a.
 \end{aligned}
\end{equation}
Thus $L_{N,k}=\widetilde L_{N,k}^a+
\sum_iF_{i,k}^{a}\cdot\nabla_{v_i}$.
The regularized operators satisfy the same identity with
$\widetilde K^{a,h}$ in place of $\widetilde K^a$.
By \Cref{lem:ft-approximation}\ref{item:ft-approximation-limits},
the modified marginals satisfy
\begin{equation}\label{eq:ft-modified-BBGKY}
 \partial_t\widetilde G_{N,k}+L_{N,k}\widetilde G_{N,k}
 -\sigma\Delta_{V_k}\widetilde G_{N,k}
 +\sum_{i=1}^k\nabla_{v_i}\cdot\left(
   \frac{N-k}{N}J_{i,k}[\widetilde G_{N,k+1}]+J_{i,k}^D\right)=0.
\end{equation}
The term with coefficient $(N-k)/N$ is absent at $k=N$.
For $k<N$, it follows from \eqref{eq:ft-remainder-definition} that
\[
 J_{i,k}[\widetilde G_{N,k+1}]-E_f(t,x_i)U_{N,k}
 =J_{i,k}[H_{N,k+1}]+J_{i,k}[h_{N,k,1}].
\]
Subtracting \eqref{eq:TV-U-PDE} from \eqref{eq:ft-modified-BBGKY}
and using \eqref{eq:ft-internal-far-force} yields
\begin{equation*}
 \partial_tH_{N,k}+\widetilde L_{N,k}^aH_{N,k}
 -\sigma\Delta_{V_k}H_{N,k}
                 +\sum_i\nabla_{v_i}\cdot\mathcal F_{i,k}=0,
\end{equation*}
where
\begin{equation*}
 \mathcal F_{i,k}=F_{i,k}^{a}H_{N,k}
 +\frac{N-k}{N}
    \left(J_{i,k}[H_{N,k+1}]+J_{i,k}[h_{N,k,1}]\right)+J_{i,k}^D.
\end{equation*}
For the regularized kernel, we have
\[
 J_{i,k}^h[U_{N,k}^{h,n}f_t]
 =(K^h*\rho_f)(t,x_i)U_{N,k}^{h,n}.
\]
The regularized equation therefore has internal operator
$\widetilde L_{N,k}^{a,h}$ and flux
\begin{equation*}
 \begin{aligned}
 \mathcal F_{i,k}^{h,n}
 =F_{i,k}^{a}H_{N,k}^{h,n}
 +\frac{N-k}{N}\Bigl(&J_{i,k}^h[H_{N,k+1}^{h,n}]
                    +J_{i,k}^h[h_{N,k,1}^{h,n}]\\
                   &+(K^h*\rho_f-E_f)(t,x_i)U_{N,k}^{h,n}\Bigr)
                    +J_{i,k}^{D,h,n}.
 \end{aligned}
\end{equation*}
At $k=N$, this is
$F_{i,N}^{a}H_{N,N}^{h,n}+D_i\widetilde G_N^{h,n}$.
The internal far-field term satisfies
\begin{equation}\label{eq:ft-internal-far-flux-bound}
 \sum_{i=1}^k
 \|F_{i,k}^{a}H_{N,k}^{h,n}\|_{L^q(\widetilde W_{N,k}^h)}^q
 \le kM_a^qY_{N,k}^{h,n}.
\end{equation}

For $k<N$, nonnegativity of $\widetilde\phi^{a,h}$,
\eqref{eq:kernel-assumptions} and
\eqref{eq:ft-core-regularization-bounds} give
\[
 \begin{gathered}
 \widetilde e_{N,k+1}^h
 \ge\widetilde e_{N,k}^h+1+|v_{k+1}|^2,\\
 \sup_{0<a\le a_2}\sup_{0<h<a/10}\|K^h\|_{q,\varpi}
 \le\|K\|_{q,\varpi}+C_{d,s,q}a_2^{\kappa/q}<\infty.
 \end{gathered}
\]
Apply \Cref{lem:common-weighted-holder} to $K^h$ with the energies
$\widetilde e_{N,k}^h$ and $\widetilde e_{N,k+1}^h$.
For the auxiliary remainder, replace $\widetilde e_{N,k+1}^h$ by
$\widetilde e_{N,k}^h+1+|v_{k+1}|^2$ and use
$\widetilde e_{N,k}^h\le e_{N,k}^h$ with
\Cref{lem:ft-two-label-residual}. We obtain
\begin{align}
 \sum_{i=1}^k\|J_{i,k}^h[H_{N,k+1}^{h,n}]
                                  \|_{L^q(\widetilde W_{N,k}^h)}^q
 &\le C_TkY_{N,k+1}^{h,n},
 \label{eq:ft-error-interaction-flux}\\
 \sum_{i=1}^k\|J_{i,k}^h[h_{N,k,1}^{h,n}]
                                  \|_{L^q(\widetilde W_{N,k}^h)}^q
 &\le C_Tk\int|h_{N,k,1}^{h,n}|^q
 e^{\lambda(t)(e_{N,k}^h+1+|v_{k+1}|^2)}\varpi_{k+1}\notag\\
 &\le D_T^k(\varepsilon_N+\xi_{n,h}).
 \label{eq:ft-remainder-interaction-flux}
\end{align}
Combining the density bound \eqref{eq:common-derived-f-envelope},
\eqref{eq:ft-core-regularization-bounds} and \eqref{eq:ft-U-cap},
we obtain
\begin{equation}\label{eq:ft-regularized-field-error}
 \begin{aligned}
 \|K^h*\rho_f-E_f\|_\infty
 &\le C_Th^{d-s-1},\\
 \sum_i\|(K^h*\rho_f-E_f)U_{N,k}^{h,n}
                         \|_{L^q(\widetilde W_{N,k}^h)}^q
 &\le D_T^kh^{q(d-s-1)}.
 \end{aligned}
\end{equation}
Using $|w_1+\cdots+w_5|^q\le5^{q-1}\sum_{\ell=1}^5|w_\ell|^q$,
combine \eqref{eq:ft-internal-far-flux-bound}--\eqref{eq:ft-regularized-field-error}
with \Cref{lem:ft-clipping-flux} to obtain
\begin{equation*}
 \begin{aligned}
 \sum_i\|\mathcal F_{i,k}^{h,n}\|_{L^q(\widetilde W_{N,k}^h)}^q
 \le C_T\Bigl[&kM_a^qY_{N,k}^{h,n}+kY_{N,k+1}^{h,n}
                      +k(1+M_a)^{2q}Y_{N,k+2}^{h,n}\\
 &+(1+M_a)^{2q}D_T^k
          (\varepsilon_N+\xi_{n,h}+h^{q(d-s-1)})\Bigr].
 \end{aligned}
\end{equation*}
At fixed $N,a,h,n$, apply \Cref{lem:common-weighted-flux-energy}
with internal operator $\widetilde L_{N,k}^{a,h}$ and weight
$\widetilde W_{N,k}^h$, using the energy tests in
\Cref{lem:ft-approximation}\ref{item:ft-approximation-limits}.
Since $\lambda(t)\ge\lambda(T)>0$, we conclude that
\begin{equation}\label{eq:ft-regularized-error-hierarchy}
 \begin{aligned}
 Y_{N,k}^{h,n}(t)-Y_{N,k}^{h,n}(s)
 \le C_T\int_s^t\Bigl[&k\bigl((1+M_a)^qY_{N,k}^{h,n}
             +Y_{N,k+1}^{h,n}+(1+M_a)^{2q}Y_{N,k+2}^{h,n}\bigr)\\
 &+(1+M_a)^{2q}D_T^k
     (\varepsilon_N+\xi_{n,h}+h^{q(d-s-1)})\Bigr](u)\,\dd u.
 \end{aligned}
\end{equation}
The constants $C_T,D_T$ are independent of $a,N,k,h,n$, and
$Y_{N,k}^{h,n}(0)=0$. At fixed $N,a$,
\Cref{lem:ft-approximation}\ref{item:ft-approximation-weights}
applied with the near-field energy yields
\[
 Y_{N,k}^{h,n}\longrightarrow Y_{N,k}^{h}\longrightarrow Y_{N,k}
 \quad\text{in }C([0,T]),\qquad
 \xi_{n,h}\longrightarrow0\quad(n\to\infty).
\]
By \eqref{eq:fixed-N-error-source-convergence} and
\Cref{lem:ft-approximation}\ref{item:ft-approximation-limits},
we may pass successively to these limits in
\eqref{eq:ft-regularized-error-hierarchy} and obtain
\eqref{eq:ft-error-hierarchy} on every $[s,t]\subset[0,T]$.
\end{proof}

\subsection{Geometric summation and Bihari's inequality}

\begin{lemma}
\label{lem:ft-osgood}
Let $J\ge0$ be a fixed integer, let $C_4,D\ge1$ be independent of
$N$, and let $\varepsilon_N>0$. Suppose that continuous nonnegative
functions $y_1,\ldots,y_N$ satisfy $y_k(0)=0$, $y_k\le C_4^k$, and
\begin{equation}\label{eq:ft-general-error-hierarchy}
 y_k(t)-y_k(s)\le\int_s^t\left[
 k\sum_{j=0}^Ja_j(u)y_{k+j}(u)
       +b(u)D^k\varepsilon_N\right]\dd u,
 \qquad 0\le s\le t\le T,
\end{equation}
where $y_k=0$ for $k>N$ and $a_0,\ldots,a_J,b$ are nonnegative
functions in $L^1(0,T)$, independent of $N$. Define
\begin{equation*}
 r=\frac1{4\max\{C_4,D\}},\qquad
 c(t)=b(t)+3\sum_{j=0}^Ja_j(t)r^{-j},\qquad
 \omega(t)=\exp\left(-\int_0^tc(u)\,\dd u\right).
\end{equation*}
Then
\begin{equation}\label{eq:ft-osgood-result}
 \sup_{0\le u\le t}y_k(u)
 \le(e^2/r)^k\varepsilon_N^{\omega(t)},
 \qquad 1\le k\le N,\quad 0\le t\le T.
\end{equation}
\end{lemma}

\begin{proof}
If $\varepsilon_N\ge1$, the conclusion follows from
$y_k\le C_4^k$ and $e^2/r\ge C_4$. Assume $0<\varepsilon_N\le1$ and set
$\mathcal Y(t)=\sum_{k=1}^Nr^ky_k(t)$.
The choice of $r$ ensures $r^ky_k\le4^{-k}$ and
$0\le\mathcal Y\le1/3$. At a time when $\mathcal Y>0$,
choose $m\ge1$ such that $4^{-m}\le\mathcal Y<4^{-(m-1)}$.
The geometric tail satisfies
\begin{equation*}
 \sum_{k>m}k4^{-k}
 =4^{-m}\left(\frac{m+1}{3}+\frac19\right)
 \le\frac{m+1}{2}\,4^{-m},\qquad
 m\le1+\frac{\log(1/\mathcal Y)}{\log4}.
\end{equation*}
It follows that
\begin{equation}\label{eq:ft-log-moment}
 \sum_{k=1}^Nkr^ky_k
 \le m\mathcal Y+\sum_{k>m}k4^{-k}
 \le\frac{3m+1}{2}\mathcal Y
 \le3\mathcal Y\log(e/\mathcal Y).
\end{equation}
At $\mathcal Y=0$, interpret $\mathcal Y\log(e/\mathcal Y)$
by continuity. For $0\le j\le J$, shifting the index yields
\begin{equation}\label{eq:ft-shifted-geometric-sum}
 \sum_{k=1}^Nkr^ky_{k+j}
 =r^{-j}\sum_{\ell=j+1}^N(\ell-j)r^\ell y_\ell
 \le r^{-j}\sum_{\ell=1}^N\ell r^\ell y_\ell.
\end{equation}
Multiply \eqref{eq:ft-general-error-hierarchy} with $s=0$ by
$r^k$ and sum over $k$. Since $rD\le1/4$,
\eqref{eq:ft-log-moment} and \eqref{eq:ft-shifted-geometric-sum}
imply
\begin{equation}\label{eq:ft-summed-hierarchy}
 \mathcal Y(t)\le\int_0^t\left[
 3\left(\sum_{j=0}^Ja_j(u)r^{-j}\right)
       \mathcal Y(u)\log\frac e{\mathcal Y(u)}
       +\frac13b(u)\varepsilon_N\right]\dd u.
\end{equation}
Since $z\mapsto z\log(e^2/z)$ is increasing on $[0,4/3]$ and
$0<\mathcal Y+\varepsilon_N\le4/3$, we have
\[
 \begin{aligned}
 \mathcal Y\log(e/\mathcal Y)
 &\le\mathcal Y\log(e^2/\mathcal Y)
 \le(\mathcal Y+\varepsilon_N)
       \log\frac{e^2}{\mathcal Y+\varepsilon_N},\\
 \frac{\varepsilon_N}{3}
 &\le(\mathcal Y+\varepsilon_N)
       \log\frac{e^2}{\mathcal Y+\varepsilon_N}.
 \end{aligned}
\]
Substituting these bounds into \eqref{eq:ft-summed-hierarchy}, we obtain
\begin{equation*}
 \mathcal Y(t)+\varepsilon_N
 \le\varepsilon_N+\int_0^tc(u)
 (\mathcal Y(u)+\varepsilon_N)
       \log\frac{e^2}{\mathcal Y(u)+\varepsilon_N}\,\dd u.
\end{equation*}
The absolutely continuous solution of $z'=c(t)z\log(e^2/z)$ with
$z(0)=\varepsilon_N$ satisfies
\begin{equation*}
 \begin{aligned}
 \frac{\dd}{\dd t}\log\frac{e^2}{z(t)}
 &=-c(t)\log\frac{e^2}{z(t)},\\
 \log\frac{e^2}{z(t)}
 &=\omega(t)\log\frac{e^2}{\varepsilon_N},\qquad
 z(t)=e^{2(1-\omega(t))}\varepsilon_N^{\omega(t)}.
 \end{aligned}
\end{equation*}
Bihari's inequality \cite{Bihari1956} implies
$\mathcal Y+\varepsilon_N\le z$ until $z$ reaches $4/3$.
For $z\ge4/3$, we have $\mathcal Y+\varepsilon_N\le4/3\le z$.
Since $z$ is nondecreasing, for $u\le t$ we obtain
\begin{equation*}
 r^ky_k(u)\le\mathcal Y(u)+\varepsilon_N
 \le z(u)\le z(t)\le e^2\varepsilon_N^{\omega(t)}.
\end{equation*}
Dividing by $r^k$ and using $e^2\le e^{2k}$ proves
\eqref{eq:ft-osgood-result}.
\end{proof}

Set $a=a_T$ and fix constants $C_T,D_T$ for which
\eqref{eq:ft-error-hierarchy} holds. With these fixed values and
$C_4$ from \eqref{eq:ft-error-cap}, define
\begin{equation}\label{eq:ft-final-parameters}
 \begin{aligned}
 r&=\frac1{4\max\{C_4,D_T\}},\\
 c_T&=C_T(1+M_{a_T})^{2q}
 +3C_T\bigl((1+M_{a_T})^q+r^{-1}
                         +(1+M_{a_T})^{2q}r^{-2}\bigr),\\
 \gamma_T&=\frac1q e^{-c_TT}.
 \end{aligned}
\end{equation}
The constants $C_4,C_T,D_T$ and $r$ are independent of the radius
$a\in(0,a_T]$. In particular, \eqref{eq:ft-far-force-size} implies
\begin{equation*}
 \begin{aligned}
 1\le c_T
 &\le C_T(4+3r^{-1}+3r^{-2})(1+M_{a_T})^{2q}\\
 &\le C_T'\bigl(1+a_T^{-2q(s+1)}\bigr).
 \end{aligned}
\end{equation*}
The constant $C_T'$ depends only on $T$ and the constants fixed
before the choice \eqref{eq:ft-explicit-radius}.
For every finite $T$ under consideration, $a_T>0$ and
$\gamma_T>0$.

\begin{proposition}
\label{prop:ft-clipped-comparison}
\label{lem:ft-velocity-weighted-TV}
With $\gamma_T$ given by \eqref{eq:ft-final-parameters}, there is $C_T\ge1$,
independent of $N,k$, such that
\begin{equation}\label{eq:ft-clipped-TV}
 \sup_{t\le T}\|\widetilde G_{N,k}(t)-f_t^{\otimes k}\|_{\TV}
 \le C_T^k\varepsilon_N^{\gamma_T},\qquad 1\le k\le N.
\end{equation}
For $k=1,2$ with $k\le N$, we also have
\begin{equation}\label{eq:ft-velocity-weighted-TV}
 \sup_{t\le T}\int (1+|V_k|^2)|\widetilde G_{N,k}(t)-f_t^{\otimes k}|
 \le C_T\varepsilon_N^{\gamma_T}.
\end{equation}
\end{proposition}

\begin{proof}
Use the fixed constants in \eqref{eq:ft-final-parameters} and
apply \Cref{lem:ft-osgood} to \eqref{eq:ft-error-hierarchy} with
$J=2$, $D=D_T$ and the constant coefficients
\[
 a_0=C_T(1+M_{a_T})^q,\qquad
 a_1=C_T,\qquad
 a_2=C_T(1+M_{a_T})^{2q},\qquad
 b=C_T(1+M_{a_T})^{2q}.
\]
By \eqref{eq:ft-error-cap}, we have
\[
 \sup_{t\le T}Y_{N,k}(t)
 \le(e^2/r)^k\varepsilon_N^{q\gamma_T}.
\]
In the remaining estimates, $C_T$ denotes a generic constant
independent of $N,k$, with $r,c_T,\gamma_T$ fixed by
\eqref{eq:ft-final-parameters}.
The lower bound $\widetilde e_{N,k}\ge\sum_i(1+|v_i|^2)$,
$\lambda(t)\ge\lambda(T)>0$ and the spatial integrability in
\eqref{eq:TV-alpha-choice} imply
\begin{equation}\label{eq:ft-inverse-weight-moments}
 \sup_{t\le T}\int \widetilde W_{N,k}^{-1/(q-1)}\le C_T^k,\qquad
 \sup_{t\le T}\int (1+|V_k|^2)^{q'}\widetilde W_{N,k}^{-1/(q-1)}\le C_T
 \quad(k=1,2).
\end{equation}
H\"older's inequality gives
\begin{align*}
 \|H_{N,k}\|_1
 &\le Y_{N,k}^{1/q}
       \left(\int \widetilde W_{N,k}^{-1/(q-1)}\right)^{1/q'}
 \le C_T^k\varepsilon_N^{\gamma_T},\\
 \int(1+|V_k|^2)|H_{N,k}|
 &\le Y_{N,k}^{1/q}
       \left(\int(1+|V_k|^2)^{q'}\widetilde W_{N,k}^{-1/(q-1)}\right)^{1/q'}
 \le C_T\varepsilon_N^{\gamma_T}\quad(k=1,2).
\end{align*}
Since $U_{N,k}$ and $f_t^{\otimes k}$ have unit mass,
\Cref{lem:TV-U-product} gives
\begin{equation*}
 \|U_{N,k}-f_t^{\otimes k}\|_1
 \le C_T^k\min\{1,\varepsilon_N\}
 \le C_T^k\varepsilon_N^{\gamma_T}.
\end{equation*}
This proves \eqref{eq:ft-clipped-TV}. For $k=1,2$,
\eqref{eq:TV-U-bound} and \eqref{eq:common-derived-f-envelope} imply
\[
 \int(1+|V_k|^2)^{q'}|U_{N,k}-f_t^{\otimes k}|\le C_T.
\]
Hence \Cref{lem:TV-U-product} and H\"older's inequality yield
\begin{align*}
 \int(1+|V_k|^2)|U_{N,k}-f_t^{\otimes k}|
 &\le\|U_{N,k}-f_t^{\otimes k}\|_1^{1/q}
 \left(\int(1+|V_k|^2)^{q'}|U_{N,k}-f_t^{\otimes k}|\right)^{1/q'}\\
 &\le C_T\min\{1,\varepsilon_N\}^{1/q}
 \le C_T\varepsilon_N^{\gamma_T}.
\end{align*}
This proves \eqref{eq:ft-velocity-weighted-TV}.
\end{proof}

%\subsection{Spatial observables}

\begin{lemma}
\label{lem:ft-spatial-observables}
For $g\in C_b^1(\mathcal D;\R)$, set
\begin{equation*}
 D_g(t)=\frac1N\sum_{i=1}^Ng(\widetilde X_i(t))
                              -\int g(x)\rho_f(t,x)\,\dd x.
\end{equation*}
There is $C_{T,g}\le C_T(\|g\|_\infty+\|\nabla g\|_\infty)^2$ such that
\begin{equation*}
 \E\sup_{0\le t\le T}|D_g(t)|^2\le C_{T,g}\varepsilon_N^{\gamma_T}.
\end{equation*}
\end{lemma}

\begin{proof}
Assume first that $N\ge2$. Test the continuity equation against
$g$. On $\R^d$, insert a spatial cutoff and remove it using the
integrable velocity moment in \eqref{eq:common-derived-f-envelope}.
We obtain
\begin{equation*}
 \frac{\dd}{\dd t}\int g\rho_f
 =\int\nabla g(x)\cdot v f_t(x,v)\,\dd x\dd v
 \quad\text{for a.e. }t\in[0,T].
\end{equation*}
Together with $\widetilde X_i'=\widetilde V_i$, this shows that
$D_g$ is almost surely absolutely continuous, with
\begin{equation}\label{eq:ft-observable-path-derivative}
 D_g'=\frac1N\sum_i\nabla g(\widetilde X_i)\cdot\widetilde V_i
                                  -\int\nabla g(x)\cdot v f_t.
\end{equation}
By \eqref{eq:common-derived-f-envelope}, we have
\[
 \left|\nabla g(x)\cdot v-\int\nabla g(y)\cdot w f_t(y,w)\,\dd y\dd w\right|^2
 \le C_T\|\nabla g\|_\infty^2(1+|v|^2).
\]
By exchangeability and \eqref{eq:ft-observable-path-derivative}, we have
\begin{align*}
 \E|D_g(t)|^2
 &=\frac1N\int\left|g(x)-\int g\rho_f(t)\right|^2
   \widetilde G_{N,1}(t,x,v)\,\dd x\dd v\\
 &\quad+\frac{N-1}{N}\int
       \prod_{\ell=1}^2\left(g(x_\ell)-\int g\rho_f(t)\right)
 \bigl(\widetilde G_{N,2}(t)-f_t^{\otimes2}\bigr)\,\dd Z_2\\
 &\le4\|g\|_\infty^2\left(
\frac1N+\|\widetilde G_{N,2}(t)-f_t^{\otimes2}\|_1\right)
  \end{align*}
  and
\begin{align*}
 \E|D_g'(t)|^2
 &=\frac1N\int\left|\nabla g(x)\cdot v-\int\nabla g\cdot vf_t\right|^2
  \widetilde G_{N,1}(t,x,v)\,\dd x\dd v\\
 &\quad+\frac{N-1}{N}\int
 \prod_{\ell=1}^2\left(\nabla g(x_\ell)\cdot v_\ell
    -\int\nabla g\cdot vf_t\right)
\bigl(\widetilde G_{N,2}(t)-f_t^{\otimes2}\bigr)\,\dd Z_2\\
 &\le C_T\|\nabla g\|_\infty^2\left(
 \frac1N\int(1+|v|^2)\widetilde G_{N,1}(t)
 +\int (1+|V_2|^2)|\widetilde G_{N,2}(t)-f_t^{\otimes2}|\right).
\end{align*}
H\"older's inequality, \eqref{eq:ft-near-cap} and
\eqref{eq:ft-inverse-weight-moments} bound the one-particle
velocity moment in the last line uniformly on $[0,T]$.
Since $N^{-1}\le\min\{1,\varepsilon_N\}\le\varepsilon_N^{\gamma_T}$,
\Cref{prop:ft-clipped-comparison} yields
\begin{equation*}
 \sup_{t\le T}\E|D_g(t)|^2
 +\mathop{\rm ess\,sup}_{t\le T}\E|D_g'(t)|^2
 \le C_{T,g}\varepsilon_N^{\gamma_T}.
\end{equation*}
Integrating the derivative of the square and taking expectations,
we obtain
\begin{align*}
 \E\sup_{t\le T}|D_g(t)|^2
 &\le\E|D_g(0)|^2+2\int_0^T\E|D_g(t)D_g'(t)|\,\dd t\\
 &\le\E|D_g(0)|^2+2\int_0^T
       (\E|D_g(t)|^2)^{1/2}(\E|D_g'(t)|^2)^{1/2}\,\dd t
 \le C_{T,g}\varepsilon_N^{\gamma_T}.
\end{align*}
For $N=1$, use $|D_g|\le2\|g\|_\infty$.

This completes the proof.
\end{proof}

\subsection{Comparison with the original dynamics}
\label{subsec:ft-unclipping}

Here $a=a_T$ is fixed. Constants in this subsection may depend on
$M_{a_T}$ and $L_{a_T}$, where
$L_a=\operatorname{Lip}(\overline K^a)$, and are independent of $N,k$.
We define
\begin{equation*}
 \mathcal E_N(t,x)=\frac1N\sum_{j=1}^N
 \overline K^a(x-\widetilde X_j(t)).
\end{equation*}
Since $\overline K^a(0)=0$, its value at $\widetilde X_i(t)$ equals
the empirical far-field force in \eqref{eq:ft-clipped-system}.

\begin{lemma}
\label{lem:ft-unclipping}
The modified system satisfies
\begin{equation}\label{eq:ft-unclipping-probability}
 \Pp\left(\sup_{t\le T}\max_{i\le N}
              |\mathcal E_N(t,\widetilde X_i(t))|>B\right)
 \le C_T\varepsilon_N^{\gamma_T}.
\end{equation}
If $Z_N$ and $\widetilde Z_N$ denote the original and modified
trajectories with initial law $G_N^0$, then
\begin{equation}\label{eq:ft-unclipping-path-TV}
 \|\operatorname{Law}(Z_N|_{[0,T]})
      -\operatorname{Law}(\widetilde Z_N|_{[0,T]})\|_{\TV}
 \le C_T\varepsilon_N^{\gamma_T}.
\end{equation}
\end{lemma}

\begin{proof}
Apply \Cref{lem:ft-spatial-observables} to each component of
$g_x(y)=\overline K^a(x-y)$. The componentwise $C_b^1$ bounds
are uniform in $x$, so
\begin{equation}\label{eq:ft-far-field-variance}
 \sup_x\E\sup_{t\le T}
           |\mathcal E_N(t,x)-m_f^a(t,x)|^2
 \le C_T\varepsilon_N^{\gamma_T}.
\end{equation}
The empirical spatial measure and the limiting spatial density
both have unit mass, so
\begin{equation}\label{eq:ft-far-field-spatial-continuity}
 |\mathcal E_N(t,x)-\mathcal E_N(t,y)|
 +|m_f^a(t,x)-m_f^a(t,y)|
 \le2L_a d_{\mathcal D}(x,y).
\end{equation}
On $\T^d$, choose a finite cover with centers
$\{x_\ell\}_{\ell=1}^J$ and radius $(8\max\{1,L_a\})^{-1}$.
The number $J$ depends only on the fixed kernel
$\overline K^{a_T}$ and the domain. By
\eqref{eq:ft-far-field-spatial-continuity}, we have
\begin{align}\label{eq:ft-covering-field-bound}
 &\sup_{t\le T,\,x\in\T^d}
       |\mathcal E_N(t,x)-m_f^a(t,x)|\notag\\
 \le&\max_{\ell\le J}\sup_{t\le T}
       |\mathcal E_N(t,x_\ell)-m_f^a(t,x_\ell)|
       +\frac{2L_a}{8\max\{1,L_a\}}\notag\\
 \le&\max_{\ell\le J}\sup_{t\le T}
       |\mathcal E_N(t,x_\ell)-m_f^a(t,x_\ell)|+\frac14.
\end{align}
Applying the union bound, Chebyshev's inequality and
\eqref{eq:ft-far-field-variance}, we obtain
\begin{align*}
 \Pp\left(\sup_{t\le T,\,x\in\T^d}
        |\mathcal E_N(t,x)-m_f^a(t,x)|>\frac12\right)
 &\le\sum_{\ell=1}^J\Pp\left(\sup_{t\le T}
        |\mathcal E_N(t,x_\ell)-m_f^a(t,x_\ell)|>\frac14\right)\\
 &\le16\sum_{\ell=1}^J\E\sup_{t\le T}
        |\mathcal E_N(t,x_\ell)-m_f^a(t,x_\ell)|^2
 \le C_T\varepsilon_N^{\gamma_T}.
\end{align*}
Together with \eqref{eq:ft-clipping-margin}, this proves
\eqref{eq:ft-unclipping-probability} on $\T^d$.

On $\R^d$, take $R\ge\max\{1,2a\}$ and $g_R\in C_b^1$ with
$0\le g_R\le1$, zero on $B_{R/2}$, equal to one outside $B_R$,
and satisfying $\|\nabla g_R\|_\infty\le C/R$.
By \eqref{eq:common-derived-f-envelope}, we have
\[
 \sup_{t\le T}\int g_R\rho_f(t)
 \le C_T\int_{|x|>R/2}\langle x\rangle^{-M}\,\dd x\le C_TR^{d-M}.
\]
For $|x|>2R$, the identity $\overline K^a=K$ outside $B_{2a}$ implies
\begin{equation}\label{eq:ft-exterior-field-bound}
 \begin{aligned}
 |\mathcal E_N(t,x)|
 &\le\frac{c_{d,s}s}{N}
        \sum_{|\widetilde X_i(t)|\le R}
                  |x-\widetilde X_i(t)|^{-s-1}
       +\frac{M_a}{N}\sum_{|\widetilde X_i(t)|>R}1\\
 &\le CR^{-s-1}+\frac{M_a}{N}\sum_i g_R(\widetilde X_i(t)).
 \end{aligned}
\end{equation}
Since $M>d$ and $s>0$, fix $R$ large enough that
$CR^{-s-1}+C_TM_aR^{d-M}\le B/2$.
This choice depends only on $a_T$, the bounds on $f$ and $B$.
It then follows from \eqref{eq:ft-exterior-field-bound} that
\[
 \sup_{t\le T,\,|x|>2R}|\mathcal E_N(t,x)|
 \le\frac B2+M_a\sup_{t\le T}|D_{g_R}(t)|.
\]
By \Cref{lem:ft-spatial-observables}, we obtain
\begin{equation}\label{eq:ft-exterior-field-cost}
 \begin{aligned}
 \Pp\left(\sup_{t\le T,\,|x|>2R}|\mathcal E_N(t,x)|>B\right)
 &\le\frac{4\max\{1,M_a\}^2}{B^2}
       \E\sup_{t\le T}|D_{g_R}(t)|^2\\
 &\le C_T\varepsilon_N^{\gamma_T}.
 \end{aligned}
\end{equation}
A finite cover of $B_{2R}$, with the calculation in
\eqref{eq:ft-covering-field-bound} restricted to that ball, yields
\begin{equation}\label{eq:ft-interior-field-cost}
 \Pp\left(\sup_{t\le T,\,|x|\le2R}
       |\mathcal E_N(t,x)-m_f^a(t,x)|>\frac12\right)
 \le C_T\varepsilon_N^{\gamma_T}.
\end{equation}
Equations \eqref{eq:ft-exterior-field-cost},
\eqref{eq:ft-interior-field-cost} and \eqref{eq:ft-clipping-margin}
prove \eqref{eq:ft-unclipping-probability}.

Couple the systems with the same initial configuration and Brownian
motions. On the event that the modified empirical field remains in
$\{|u|\le B\}$, \eqref{eq:ft-clipped-system} agrees with the
original equation. Pathwise uniqueness from \Cref{prop:common-finite-N} and
\Cref{lem:ft-approximation}\ref{item:ft-approximation-dynamics} implies
\begin{equation*}
 \{Z_N|_{[0,T]}\ne\widetilde Z_N|_{[0,T]}\}
 \subset
 \left\{\sup_{t\le T}\max_{i\le N}
       |\mathcal E_N(t,\widetilde X_i(t))|>B\right\}
 \quad\text{almost surely}.
\end{equation*}
By \eqref{eq:intro-TV-contraction} and
\eqref{eq:ft-unclipping-probability}, we obtain
\eqref{eq:ft-unclipping-path-TV}.
\end{proof}

\begin{proof}[Proof of \Cref{thm:fixed-time-riesz}]
Total variation contracts under time evaluation and marginalization.
Thus \Cref{lem:ft-unclipping,prop:ft-clipped-comparison} imply
\begin{equation*}
 \begin{aligned}
 \sup_{t\le T}\|G_{N,k}(t)-f_t^{\otimes k}\|_{\TV}
 &\le\sup_{t\le T}\|G_{N,k}(t)-\widetilde G_{N,k}(t)\|_{\TV}
      +\sup_{t\le T}\|\widetilde G_{N,k}(t)-f_t^{\otimes k}\|_{\TV}\\
 &\le C_T^k\varepsilon_N^{\gamma_T}
 \le C_T^kN^{-\gamma_T},
 \end{aligned}
\end{equation*}
after increasing $C_T$, since $\varepsilon_N=2/N$.
Total variation contraction and
\eqref{eq:common-riesz-conditioning-cost} yield
\begin{align*}
 \sup_{t\le T}\|F_{N,k}(t)-f_t^{\otimes k}\|_{\TV}
 &\le\|f_0^{\otimes N}-G_N^0\|_{\TV}
       +\sup_{t\le T}\|G_{N,k}(t)-f_t^{\otimes k}\|_{\TV}\\
 &\le C_T^k\bigl(N^{-(d/s-2)}+N^{-\gamma_T}\bigr).
\end{align*}
This completes the proof.
\end{proof}

\begin{proof}[Proof of \Cref{cor:global-coulomb}]
By the initial norm in \eqref{eq:whole-space-classical-data-norm} and
Leibniz's rule, we have
\begin{equation*}
 \sum_{|\xi|+|\eta|\le2}
 \left|\partial_x^\xi\partial_v^\eta
  \bigl(\langle v\rangle^8f_0(x,v)\bigr)\right|
 \le C\|f_0\|_{C^m_{M,a_0}}\langle x\rangle^{-M}\langle v\rangle^8e^{-a_0|v|^2}.
\end{equation*}
Since $M>3$ and $a_0>0$, the right-hand side belongs to
$L^1(\R^6)\cap L^\infty(\R^6)$. Therefore
\[
 f_0\in W^{2,1}(\R^6)\cap W^{2,\infty}(\R^6),\qquad
 \langle v\rangle^8f_0\in W^{2,1}(\R^6)\cap W^{2,\infty}(\R^6).
\]
The global classical theory of \cite{Bouchut1993}, in the formulation
of \cite[Theorem~1.1 and Remark~1.1, equations~(15)--(17)]{HuangLiuPickl2020}
with velocity weight exponent $8>6$ and derivative order two,
applies to $f_0$. For every finite $T$, the solution is nonnegative,
has unit mass, and satisfies
\[
 E_f\in C([0,T];W^{1,\infty}(\R^3)),\qquad
 \sup_{0\le t\le T}
 \|\langle v\rangle^8f_t\|_{W^{2,1}(\R^6)\cap W^{2,\infty}(\R^6)}
 <\infty.
\]
For $|\xi|\le2$ and $0\le t\le T$, we have
\begin{align*}
 \|\partial_x^\xi\rho_f(t)\|_1
             +\|\partial_x^\xi\rho_f(t)\|_\infty
 &\le\|\partial_x^\xi f_t\|_1
   +\left(\int_{\R^3}\langle v\rangle^{-8}\,\dd v\right)
             \|\langle v\rangle^8\partial_x^\xi f_t\|_\infty\\
 &\le C\|\langle v\rangle^8f_t\|_{W^{2,1}\cap W^{2,\infty}}
 \le C_T.
\end{align*}
Since $K\in L^1(B_1)$ and $K\in L^\infty(B_1^c)$, we have
\begin{equation}\label{eq:intro-global-field-regularity}
 \begin{aligned}
 \|\partial_x^\xi E_f(t)\|_\infty
 \le\|K\|_{L^1(B_1)}\|\partial_x^\xi\rho_f(t)\|_\infty
    +\|K\|_{L^\infty(B_1^c)}\|\partial_x^\xi\rho_f(t)\|_1
 \le C_T,\qquad |\xi|\le2.
 \end{aligned}
\end{equation}
Equation \eqref{eq:intro-global-field-regularity} implies
$E_f(t)\in W^{2,\infty}\subset C_b^1$.
The continuity in $W^{1,\infty}$ then yields
\eqref{eq:intro-field-regularity}. The velocity derivative satisfies
\begin{equation*}
 \int_0^T\|\nabla_vf(t)\|_1\,\dd t
 \le T\sup_{0\le t\le T}\|f(t)\|_{W^{1,1}}<\infty.
\end{equation*}
The initial norm implies \eqref{eq:intro-initial-decay} with
$C_0=\|f_0\|_{C^m_{M,a_0}}$.
Hence \Cref{ass:limiting-solution} holds on every finite interval.
Apply \Cref{thm:fixed-time-riesz} with $d=3$ and $s=1$.
Since $q>3$ and $\gamma_T\le q^{-1}<1$, the conditioning error
$N^{-1}$ is bounded by $N^{-\gamma_T}$, which proves
\eqref{eq:intro-global-coulomb}.
\end{proof}

\section{Particle dynamics and force regularization}
\label{sec:kernel-verification}

We recall the force regularization of \cite[Section~4]{JiangWuShortTime}
and prove well-posedness and convergence for particle systems with
bounded fields depending on all particle positions.

\subsection{Force regularization}
\label{subsec:whole-space-kernel-class}

\begin{lemma}[{\cite[Lemma~4.1]{JiangWuShortTime}}]
\label{lem:common-riesz-regularization}
Let $d\ge2$, $0<s<d-1$, and $\mathcal D\in\{\T^d,\R^d\}$.
Let $\phi=\phi_{\mathcal D,s}$,
$K=-\nabla\phi$ be as in the introduction.  There exist
$\varepsilon_0\in(0,1)$, $r_0>0$, and smooth even potentials
$\phi_\varepsilon\ge1$
such that
\begin{equation*}
 \phi_\varepsilon=\phi\ \text{when }d_{\mathcal D}(x,0)\ge2\varepsilon,
 \quad 0\le\phi_\varepsilon\le C\phi,
 \quad K_\varepsilon=-\nabla\phi_\varepsilon,
 \quad \operatorname{supp}(K_\varepsilon-K)
 \subset\{d_{\mathcal D}(x,0)\le2\varepsilon\}.
\end{equation*}
Here $d_{\mathcal D}$ is geodesic distance on the torus and Euclidean
distance on $\R^d$.  The constants are independent of
$0<\varepsilon<\varepsilon_0$.  The forces satisfy
\begin{equation*}
 |K_\varepsilon(x)|+|K(x)|\le C d_{\mathcal D}(x,0)^{-s-1}
 \quad(0<d_{\mathcal D}(x,0)<r_0).
\end{equation*}
For $1\le p<d/(s+1)$, we have
\begin{equation*}
 \|K_\varepsilon-K\|_{L^p(\mathcal D)}
 \le C_p\varepsilon^{d/p-(s+1)}.
\end{equation*}
On the torus,
$\sup_{0<\varepsilon<\varepsilon_0}\|K_\varepsilon\|_{L^p}<\infty$ in this range.
On $\R^d$, take $\varpi(x)=\langle x\rangle^\alpha$, $\alpha>0$,
in \eqref{eq:common-kernel-norm}.
If $q>1$ satisfies $(s+1)q'<d$, then
\begin{equation*}
 \|K_\varepsilon-K\|_{q,\varpi}
 \le C\varepsilon^{d/q'-(s+1)}\longrightarrow0.
\end{equation*}
If in addition $\alpha>d(q-1)$, then
\begin{equation*}
 \sup_{0<\varepsilon<\varepsilon_0}\|K_\varepsilon\|_{q,\varpi}<\infty.
\end{equation*}
For each fixed $\varepsilon$, $K_\varepsilon$ is smooth and globally
Lipschitz.
\end{lemma}

\subsection{Dynamics with a configuration-dependent field}
\label{sec:whole-space-dynamics}

\begin{proposition}\label{prop:common-finite-N}
Let $N\ge k\ge1$, $T>0$, and
$b_i\in C([0,T];C_b^1(\mathcal D^k;\R^d))$. Let $\psi$ be an even,
nonnegative potential, smooth away from zero, with
$\psi(x)\to\infty$ as $d_{\mathcal D}(x,0)\downarrow0$. Suppose that
$K^\psi=-\nabla\psi$ is bounded and Lipschitz on every set
$\{d_{\mathcal D}(x,0)\ge r\}$, $r>0$.
For every deterministic collision-free initial configuration, the system
\begin{equation*}
 \dd X_i=V_i\,\dd t,\qquad
 \dd V_i=\left(\frac1N\sum_{\substack{1\le j\le k\\j\ne i}}K^\psi(X_i-X_j)
                         +b_i(t,X_k)\right)\dd t
                         +\sqrt{2\sigma}\,\dd B_i,
 \qquad 1\le i\le k,
\end{equation*}
has a pathwise unique strong solution on $[0,T]$, with distinct
positions throughout this interval. The same conclusion holds for every
initial law supported on collision-free configurations and independent
of the Brownian motions. The transition
operators preserve mass and contract total variation. In particular,
$k=N$, $\psi=\phi_{\mathcal D,s}$ and $b_i=0$ give the original
particle system.
\end{proposition}

\begin{proof}
Local existence and pathwise uniqueness hold away from the collision
diagonals. As in \cite[Proposition~4.2]{JiangWuShortTime}, we estimate
the particle energy. Set
\begin{equation*}
 \mathcal H_{N,k}
 =\frac12\sum_{i=1}^k|V_i|^2
        +\frac1N\sum_{i<j\le k}\psi(X_i-X_j),
 \qquad L_T=\max_i\|b_i\|_{L^\infty([0,T]\times\mathcal D^k)}.
\end{equation*}
The internal forces cancel in It\^o's formula, so
\begin{equation}\label{eq:common-forced-energy-ito}
 \dd\mathcal H_{N,k}
 =\left(d\sigma k+\sum_iV_i\cdot b_i(t,X_k)\right)\dd t
                     +\sqrt{2\sigma}\sum_iV_i\cdot\dd B_i,
 \qquad
 \left|\sum_iV_i\cdot b_i\right|
 \le\mathcal H_{N,k}+\frac{kL_T^2}{2}.
\end{equation}
Let $\tau_L=\inf\{t:\mathcal H_{N,k}(t)\ge L\}$. Before $\tau_L$, we have
\begin{equation*}
 |V_i|\le\sqrt{2L},\qquad
 \psi(X_i-X_j)\le NL,\qquad
 |X_i(t)-X_i(0)|\le T\sqrt{2L},
\end{equation*}
where the last estimate concerns $\R^d$. Hence the trajectory remains
in a compact subset of the collision-free phase space. For deterministic
data, choose $L>\mathcal H_{N,k}(0)$. The stopped martingale has
quadratic variation at most $4\sigma LT$. For $t\le T$, taking
expectations in \eqref{eq:common-forced-energy-ito} and applying
Gronwall's inequality yields
\begin{align}\label{eq:common-stopped-energy}
 \E\mathcal H_{N,k}(t\wedge\tau_L)
 &\le\E\mathcal H_{N,k}(0)
       +k\left(d\sigma+\frac{L_T^2}{2}\right)t
       +\int_0^t\E\mathcal H_{N,k}(u\wedge\tau_L)\,\dd u\notag\\
 &\le e^t\E\mathcal H_{N,k}(0)
       +k\left(d\sigma+\frac{L_T^2}{2}\right)(e^t-1),\notag\\
 \Pp(\tau_L\le T)
 &\le\frac1L\E\mathcal H_{N,k}(T\wedge\tau_L)
 \le\frac{C_T}{L}\bigl(\E\mathcal H_{N,k}(0)+k\bigr).
\end{align}
Letting $L\to\infty$ proves continuation and noncollision for
deterministic collision-free data. For an arbitrary initial law $\mu_k$
supported on collision-free configurations, condition
\eqref{eq:common-stopped-energy} on the initial configuration.
For $0<R<L$, we obtain
\begin{equation*}
 \Pp_{\mu_k}(\tau_L\le T)
 \le\mu_k\{\mathcal H_{N,k}(0)>R\}
       +\frac{C_T(R+k)}{L}.
\end{equation*}
First let $L\to\infty$, then $R\to\infty$.
Since $\mathcal H_{N,k}(0)<\infty$ almost surely, the explosion
probability is zero.
Pathwise uniqueness yields a time-inhomogeneous Markov transition family,
whose conservativity ensures mass preservation and total variation contraction.
\end{proof}

\begin{proposition}\label{prop:common-regularization-removal}
Under the hypotheses of \Cref{prop:common-finite-N}, let $K^{\psi,h}$
be smooth globally Lipschitz forces agreeing with $K^\psi$ outside
$B_{2h}$. Let $\mu_k$ be an initial law supported on collision-free
configurations. Couple the regularized and singular systems with the
same initial configuration of law $\mu_k$ and the same Brownian motions,
taking the initial configuration independently of the Brownian motions.
Keep the field $b_i$ unchanged.
Then
\begin{equation}\label{eq:whole-space-path-TV}
 \|\operatorname{Law}(Z_k^h|_{[0,T]})
       -\operatorname{Law}(Z_k|_{[0,T]})\|_{\TV}
 \le2\Pp_{\mu_k}(d_T\le2h)\longrightarrow0,
 \qquad
 d_T=\min_{t\le T}\min_{i<j}d_{\mathcal D}(X_i(t),X_j(t)),
\end{equation}
with $d_T=\infty$ for $k=1$. If $\mu_k$ has density $g$, the forward
laws have densities $F_k^h,F_k\in C([0,T];L^1)$ and
\begin{equation}\label{eq:whole-space-regularization-removal-L1}
 \sup_{t\le T}\|F_k^h(t)-F_k(t)\|_1\longrightarrow0.
\end{equation}
\end{proposition}

\begin{proof}
Use the synchronous coupling of \cite[Proposition~4.3]{JiangWuShortTime}.
Let $\widehat\tau_h$ be the first time that either process enters
$\{\min_{i<j}d_{\mathcal D}(x_i,x_j)\le2h\}$. Equality of the
coefficients outside this set, pathwise uniqueness, and continuity imply
\begin{equation*}
 Z_k^h(t)=Z_k(t)\quad(0\le t\le T\wedge\widehat\tau_h),\qquad
 \Pp(Z_k^h|_{[0,T]}\ne Z_k|_{[0,T]})
          \le\Pp(d_T\le2h).
\end{equation*}
By \Cref{prop:common-finite-N}, $d_T>0$ almost surely. This proves
\eqref{eq:whole-space-path-TV}.

For fixed $h$, subtracting the Brownian velocity path yields a globally
Lipschitz $C^1$ flow. Since the configuration field is independent of
velocity, the resulting stochastic flow $\Phi_t^h$ satisfies
\begin{equation}\label{eq:common-flow-jacobian}
\begin{aligned}
 \frac{\dd}{\dd t}\det D\Phi_t^h
 &=\left[\sum_i\nabla_{x_i}\cdot v_i+
 \sum_i\nabla_{v_i}\cdot\left(
     \frac1N\sum_{j\ne i}K^{\psi,h}(x_i-x_j)+b_i(t,X_k)\right)\right]
       (\Phi_t^h)\det D\Phi_t^h=0,\\
 \det D\Phi_0^h&=1.
\end{aligned}
\end{equation}
Thus $\det D\Phi_t^h=1$, and the density is given by
\begin{equation}\label{eq:common-flow-density}
 F_k^h(t)=\E[g\circ(\Phi_t^h)^{-1}],\qquad
 \|g\circ(\Phi_t^h)^{-1}\|_1=\|g\|_1.
\end{equation}
The $C_c$ approximation argument in
\cite[Proposition~4.3]{JiangWuShortTime}, applied to
\eqref{eq:common-flow-density}, proves $F_k^h\in C_tL^1$ for every
$g\in L^1$. By \eqref{eq:whole-space-path-TV} and contraction under
evaluation, we have
\begin{equation*}
 \sup_{t\le T}\|F_k^h(t)-F_k^{h'}(t)\|_1
 \le2\Pp_{\mu_k}(d_T\le2h)+2\Pp_{\mu_k}(d_T\le2h')
 \longrightarrow0.
\end{equation*}
The limit belongs to $C([0,T];L^1)$ and has the law of the singular
process at every time, proving
\eqref{eq:whole-space-regularization-removal-L1}.
\end{proof}

\begin{remark}\label{rem:general-particle-potential}
The preceding results apply to the full and near-field potentials used in
Section~\ref{sec:fixed-time-riesz}. Their constants in the stopped
energy estimate depend on the field only through $L_T$.
\end{remark}

\appendix
\section{Approximation at fixed particle number}
\label{app:fixed-N-passage}

We use the conditioned initial density $G_N^0$ from
\eqref{eq:comparison-initial-envelope} and the exponents fixed in
Section~\ref{sec:common-comparison}. Following
\cite[Appendix~A]{JiangWuShortTime}, we take all limits at fixed $N$
and, for the modified system, at fixed splitting radius $a$.
We first remove the phase-space cutoff, then remove the smoothing
of the prescribed field, and finally take the initial-data and force
limits in that order. The pointwise bounds with prefactor $C_T^N$,
obtained from \Cref{lem:common-exponential-envelope}, provide
integrable majorants for these fixed-$N$ limits. Together with the
dynamics constructed in Section~\ref{sec:kernel-verification}, they
justify the energy tests and convergence used in
Section~\ref{sec:fixed-time-riesz}. That section derives the scalar
inequalities for the regularized equations and proves the bounds
with prefactor $C_T^k$ that are uniform in the particle number.

\subsection{Initial data and comparison}
\label{subsec:common-smooth-data}

\begin{lemma}[{\cite[Lemma~A.1]{JiangWuShortTime}}]
\label{lem:common-smooth-initial-data}
Assume \eqref{eq:comparison-initial-envelope} and the exponent choices
$\beta_0<\beta_1$ of Section~\ref{sec:common-comparison}.
For fixed $N,\varepsilon$, there are smooth symmetric probability
densities $G_N^{0,\varepsilon,n}$, compactly supported in the noncompact
variables, such that
\begin{equation*}
 \xi_{n,\varepsilon}:=\|G_N^{0,\varepsilon,n}-G_N^0\|_1\longrightarrow0,
 \qquad
 G_{N,k}^{0,\varepsilon,n}
 \le A_2^ke^{-\beta_0e_{N,k}^\varepsilon}
               \prod_{j=1}^k\langle x_j\rangle^{-M_2}.
\end{equation*}
The spatial factors are omitted on $\T^d$, and $A_2$ is independent
of $N,k,n,\varepsilon$.
\end{lemma}

\begin{proof}
Apply \cite[Lemma~A.1]{JiangWuShortTime} with spatial exponent $M_2$
and initial bound \eqref{eq:comparison-initial-regularized-envelope}.
The constant in that lemma depends only on $A,M_2,\beta_0,\beta_1$.
\end{proof}

\begin{remark}\label{rem:smoothing-general-potential}
The same conclusion holds for a smooth nonnegative pair potential
$\psi$ with bounded gradient when the initial bound holds with
$e_{N,k}^{\psi}$ and exponent $\beta_1$. Indeed, choose the convolution
radii in \cite[Lemma~A.1]{JiangWuShortTime} so that
\[
 0<\delta_n\le n^{-1},\qquad
 2\|\nabla\psi\|_\infty\delta_n\le1.
\]
For $|y_i|\le\delta_n$, the potential part satisfies
\begin{equation}\label{eq:common-smoothing-potential}
 \frac1N\left|\sum_{i\ne j\le k}
 [\psi(x_i-x_j)-\psi(x_i-y_i-x_j+y_j)]\right|
 \le2\|\nabla\psi\|_\infty\delta_n\frac{k(k-1)}N\le k.
\end{equation}
For $|w|,|y|\le\delta_n\le1$, we have
\[
 \beta_0|v|^2-\beta_1|v-w|^2
 \le\frac{\beta_0\beta_1}{\beta_1-\beta_0},\qquad
 \langle x-y\rangle^{-M_2}\le2^{M_2}\langle x\rangle^{-M_2}.
\]
The normalization of the truncated density in
\cite[Lemma~A.1]{JiangWuShortTime} contributes a factor at most $2$.
Together with \eqref{eq:common-smoothing-potential}, these estimates give
\[
 A_2=2^{M_2+1}A
       \exp\left(\beta_0+\frac{\beta_0\beta_1}{\beta_1-\beta_0}\right)
 \quad\text{on }\R^d.
\]
On $\T^d$, take the same constant without the factor $2^{M_2}$.
Thus $A_2$ is independent of $\psi$ and the convolution radii.
\end{remark}

Identity \eqref{eq:TV-level-N-equality} also holds for the regularized
particle laws and auxiliary solutions started from these common smooth data.
Apply \Cref{lem:common-exponential-envelope} to the auxiliary equation
with $A_1=A_2$ and $\beta_3=\beta_0$.
At level $N$, use
\eqref{eq:TV-level-N-equality} and integrate the last $N-k$ variables.
After enlarging $A_T,B_T$, we obtain
\begin{equation*}
	U_{N,k}^{\varepsilon,n}(t,Z_k)
	\le A_T^k e^{-\beta_Te_{N,k}^\varepsilon(Z_k)}
	\prod_{j=1}^{k}\langle x_j\rangle^{-M_2}
\end{equation*}
and
\begin{equation*}
	G_{N,k}^{\varepsilon,n}(t,Z_k)
	\le B_T^N e^{-\beta_Te_{N,k}^\varepsilon(Z_k)}
	\prod_{j=1}^{k}\langle x_j\rangle^{-M_2}.
\end{equation*}
On $\T^d$ the spatial products are omitted.  The constants
$A_T,B_T,\beta_T$ are independent of $N,k,n,\varepsilon$.

We extend \cite[Lemma~A.2]{JiangWuShortTime} to fields depending on
all particle positions.

\begin{lemma}\label{lem:common-L1-stability}
Fix $N\ge k\ge1$, $0<\varepsilon<\varepsilon_0$, and $T>0$.
Let $b_i\in C([0,T];C_b^1(\mathcal D^k;\R^d))$ be independent of
velocity, and set
\[
 \mathcal P=\partial_t+L_{N,k}^\varepsilon
       +\sum_i b_i(t,X_k)\cdot\nabla_{v_i}-\sigma\Delta_{V_k}.
\]
Let $g\ge0$ belong to $L^1$, and let $U\in C([0,T];L^1)$ be its
forward evolution under $\mathcal P U=0$.
Suppose $u\in C([0,T];L^1)$ satisfies $\mathcal P u=r$ in distributions,
with its $L^1$ initial trace, and
$r,\nabla_{V_k}u\in L^1(0,T;L^1)$.
Then, for every $t\in[0,T]$, we have
\begin{equation}\label{eq:common-L1-stability}
 \sup_{0\le a\le t}\|U(a)-u(a)\|_1
 \le\|g-u(0)\|_1+\int_0^t\|r(s)\|_1\,\dd s.
\end{equation}
Writing $a_+=\max\{a,0\}$, one also has
\begin{equation}\label{eq:common-L1-order}
 \sup_{0\le a\le t}\|(U(a)-u(a))_+\|_1
 \le\|(g-u(0))_+\|_1+\int_0^t\|(-r(s))_+\|_1\,\dd s.
\end{equation}
The same conclusions hold with $K_\varepsilon$ replaced by any bounded
$C^1$ force with bounded first derivatives.
\end{lemma}

\begin{proof}
Set
\[
 a_i(t,X_k)=\frac1N\sum_{j\ne i}K_\varepsilon(x_i-x_j)+b_i(t,X_k).
\]
Then $a_i\in C_tC_b^1$ and $\nabla_{v_i}\cdot a_i=0$. These are the
properties of the field used in \cite[Lemma~A.2]{JiangWuShortTime}.
Extend $a_i$ constantly in time
outside $[0,T]$ and mollify. The resulting fields satisfy
\begin{equation}\label{eq:common-field-mollification}
 \max_i\|a_i^\nu-a_i\|_\infty\longrightarrow0,\qquad
 \sup_\nu\max_i\|a_i^\nu\|_{C_tC_b^1}<\infty.
\end{equation}
Let $U^\nu$ be the forward evolution from $g$ for the acceleration
$a^\nu$, and put
\begin{equation*}
 r^\nu=r+\sum_i(a_i^\nu-a_i)\cdot\nabla_{v_i}u,
 \qquad \|r^\nu-r\|_{L_t^1L^1}\longrightarrow0.
\end{equation*}
For a smooth terminal function $\psi$ compactly supported in the
noncompact variables, let $\varphi$ solve the backward Kolmogorov
equation with acceleration $a^\nu$ and terminal value $\psi$ at time $t$.
The representation in \cite[Lemma~A.2]{JiangWuShortTime} implies
$\|\varphi\|_\infty\le\|\psi\|_\infty$ and $\varphi\ge0$ when $\psi\ge0$.
Choose a smooth cutoff $0\le\zeta_R\le1$ equal
to one for $|V_k|\le R$ and, on $\R^d$, $|X_k|\le R^2$, and supported
where $|V_k|\le2R$ and $|X_k|\le2R^2$. On the torus only the velocity
cutoff is needed. For fixed $\nu$, the first and second derivatives of
$\varphi$ are bounded. The cutoff satisfies
\begin{equation}\label{eq:common-duality-cutoff}
 \left|\sum_i v_i\cdot\nabla_{x_i}\zeta_R\right|
 +|\nabla_{V_k}\zeta_R|\le C_{N,k}R^{-1},\qquad
 |\Delta_{V_k}\zeta_R|\le C_{N,k}R^{-2}.
\end{equation}
Use $\zeta_R\varphi$ with a smooth time cutoff on $(0,t)$ in the
duality calculation of \cite[Lemma~A.2]{JiangWuShortTime}.
The $C_tL^1$ continuity of $U^\nu,u$ justifies the endpoint terms at $0$ and $t$.
By \eqref{eq:common-duality-cutoff} and $u,r^\nu\in L_t^1L^1$,
letting $R\to\infty$ yields
\begin{equation}\label{eq:auxiliary-duality-identity}
 \langle U^\nu(t)-u(t),\psi\rangle
 =\langle g-u(0),\varphi(0)\rangle
       -\int_0^t\langle r^\nu(s),\varphi(s)\rangle\,\dd s.
\end{equation}
The flows with accelerations $a^\nu,a$ preserve volume by
\eqref{eq:common-flow-jacobian}. For $g_m\in C_c^\infty$ with
$\|g_m-g\|_1\to0$, Gronwall's inequality and the approximation
argument of \cite[Lemma~A.2]{JiangWuShortTime} imply
\begin{equation}\label{eq:common-mollified-forward-L1}
 \sup_{t\le T}\|U^\nu(t)-U(t)\|_1
 \le2\|g-g_m\|_1+C_{N,k,\varepsilon,T,m}\max_i\|a_i^\nu-a_i\|_\infty.
\end{equation}
The constant depends on the common $C_b^1$ bound in
\eqref{eq:common-field-mollification}. Letting $\nu\to0$ and then
$m\to\infty$, we obtain $U^\nu\to U$ in $C_tL^1$.
Taking the supremum in \eqref{eq:auxiliary-duality-identity}
over $|\psi|\le1$, respectively $0\le\psi\le1$, proves
\eqref{eq:common-L1-stability} and \eqref{eq:common-L1-order}.
\end{proof}

\subsection{Distributional limit of the hierarchy}

\begin{proposition}\label{prop:common-singular-hierarchy}
Fix $N$ and $q$ as in \eqref{eq:intro-q-range}.  Suppose that the
regularized particle marginals of Section~\ref{sec:common-comparison}
satisfy the following convergence at every level:
\begin{equation}\label{eq:common-singular-hierarchy-convergence}
 G_{N,k}^\varepsilon\longrightarrow G_{N,k}
 \quad\text{in }C([0,T];L^1),
\end{equation}
and that, for some continuous $\lambda:[0,T]\to(0,\infty)$, we have
\begin{equation}\label{eq:common-singular-hierarchy-weight}
 \operatorname*{ess\,sup}_{t\le T}
 \int|G_{N,k}^\varepsilon(t)|^q e^{\lambda(t)e_{N,k}^\varepsilon}
      \varpi_k\,\dd Z_k\le C_k,
\end{equation}
with $C_k$ independent of $\varepsilon$.  Then
\begin{equation*}
 \partial_tG_{N,k}+L_{N,k}G_{N,k}-\sigma\Delta_{V_k}G_{N,k}
 +\frac{N-k}{N}\sum_{i=1}^k\nabla_{v_i}\cdot J_{i,k}[G_{N,k+1}]=0
\end{equation*}
in distributions, with initial trace $G_{N,k}^0$.  The last term is
omitted at $k=N$.

If \eqref{eq:common-singular-hierarchy-convergence} and
\eqref{eq:common-singular-hierarchy-weight} hold instead for the auxiliary
densities $U_{N,k}^\varepsilon,U_{N,k}$, then
\begin{equation*}
 \partial_tU_{N,k}+L_{N,k}U_{N,k}-\sigma\Delta_{V_k}U_{N,k}
 +\frac{N-k}{N}\sum_{i=1}^k E_f(t,x_i)\cdot\nabla_{v_i}U_{N,k}=0
\end{equation*}
in distributions, with the same initial trace $G_{N,k}^0$.
Both assertions also hold for the near-field interaction and its
regularizations defined below. For particle laws satisfying
\eqref{eq:common-singular-hierarchy-convergence} with an additional
fixed bounded measurable permutation-equivariant configuration field
$b_i^N(t,X_N)$, its contribution in the limiting equation is
$\sum_{i=1}^k\nabla_{v_i}\cdot\int b_i^NG_N\,\dd Z_{k+1:N}$.
\end{proposition}

\begin{proof}
For the full interaction, the conclusions follow from
\cite[Proposition~A.3]{JiangWuShortTime}, with hypotheses
\eqref{eq:common-singular-hierarchy-convergence}--\eqref{eq:common-singular-hierarchy-weight}.
For the near-field interaction, \eqref{eq:ft-core-regularization-bounds}
and \eqref{eq:ft-core-kernel-difference} imply
\[
 \sup_{0<h<a/10}\|\widetilde K^{a,h}\|_{q,\varpi}<\infty,
 \qquad
 \|\widetilde K^{a,h}-\widetilde K^a\|_{q,\varpi}
 \le\|K^h-K\|_{L^{q'}}\longrightarrow0.
\]
Together with
$\widetilde e_{N,k+1}^h\ge\widetilde e_{N,k}^h+1+|v_{k+1}|^2$,
these bounds verify the hypotheses used in the proof of
\cite[Proposition~A.3]{JiangWuShortTime}.
For the bounded configuration field, we have
\begin{equation}\label{eq:bounded-configuration-flux-convergence}
 \sup_{t\le T}\left\|\int b_i^N(t,X_N)(G_N^h-G_N)
                         \,\dd Z_{k+1:N}\right\|_1
 \le\|b_i^N\|_\infty\sup_{t\le T}\|G_N^h-G_N\|_1\longrightarrow0.
\end{equation}
This completes the proof.
\end{proof}

\subsection{Localization and endpoint values}
\label{subsec:common-cutoff-testing}

We follow \cite[Lemma~3.6 and Section~A.4]{JiangWuShortTime}.
Fix the force and initial-data regularizations. Extend $E_f$ constantly
beyond the time endpoints and mollify in time and space. Then
\begin{equation*}
 E_f^\delta\longrightarrow E_f\quad\text{in }C_tL^\infty,
 \qquad \sup_\delta\|E_f^\delta\|_{C_tC_b^1}<\infty.
\end{equation*}
At fixed force, initial-data and field regularization parameters,
the inverse-flow and marginalization estimates in
\cite[Section~A.4, (A.11)--(A.12)]{JiangWuShortTime}
establish local smoothness of the regularized densities, differences and
interaction fluxes.
Let $e$ denote either the near-field or the full energy, and set
$W(t)=e^{\lambda(t)e}\varpi_k$. The pointwise estimates imply
\begin{equation}\label{eq:fixed-N-localization-majorant}
 (1+|V_k|^2)|H(t)|^qW(t)
 \le C_{N,k,T}(1+|V_k|^2)
 e^{-(q\beta_T-\lambda_0)\sum_i(1+|v_i|^2)}
 \prod_i\langle x_i\rangle^{\alpha-qM_2}\in L^1,
\end{equation}
uniformly in time and in $\delta$. The constant may depend on the
fixed splitting radius; spatial factors are omitted on $\T^d$.

For the smooth marginal or difference equation, write $L$ for the
regularized internal transport operator and $\mathcal F_i$ for its
velocity flux, so that
\[
 \partial_tH+LH-\sigma\Delta_{V_k}H
       +\sum_i\nabla_{v_i}\cdot\mathcal F_i=0.
\]
Use the cutoffs $\zeta_R$ from the proof of
\Cref{lem:common-L1-stability}. By
\eqref{eq:common-duality-cutoff} and boundedness of the regularized
acceleration, they satisfy
\begin{equation*}
 |L\zeta_R|+|\nabla_{V_k}\zeta_R|\le\frac{C_{N,k}}R,
 \qquad |\Delta_{V_k}\zeta_R|\le\frac{C_{N,k}}{R^2}.
\end{equation*}
The constants may depend on the fixed force regularization.
Testing with $q|H|^{q-2}H W\zeta_R$, we obtain the following identity;
for noninteger $q\ge2$, use the regularized tests in
\cite[Section~A.4]{JiangWuShortTime}.
\begin{equation}\label{eq:fixed-N-localized-energy-identity}
	\begin{aligned}
 &\frac{\dd}{\dd t}\int |H|^qW\zeta_R
 +q(q-1)\sigma\sum_i\int
        |H|^{q-2}|\nabla_{v_i}H|^2W\zeta_R\notag\\
 =&\int |H|^q(\partial_tW+LW+\sigma\Delta_{V_k}W)\zeta_R+q(q-1)\sum_i\int |H|^{q-2}\mathcal F_i\cdot\nabla_{v_i}H W\zeta_R\notag\\
 &+q\sum_i\int |H|^{q-2}H\mathcal F_i\cdot
  \nabla_{v_i}W\zeta_R+\mathcal R_R,
 \end{aligned}
\end{equation}
where
\begin{align*}
 \mathcal R_R
 =\int |H|^q\left(WL\zeta_R
       +\sigma W\Delta_{V_k}\zeta_R
       +2\sigma\nabla_{V_k}W\cdot\nabla_{V_k}\zeta_R\right)+q\sum_i\int |H|^{q-2}H\mathcal F_i\cdot \nabla_{v_i}\zeta_R W.
\end{align*}
Since $\nabla_{V_k}W=2\lambda V_kW$, the cutoff bounds and
the weighted H\"older inequality yield
\begin{align*}
 \int_s^t|\mathcal R_R(u)|\,\dd u
 &\le\frac{C_{N,k,T}}R\left[
 \int_s^t\!\int(1+|V_k|^2)|H|^qW\,\dd Z_k\dd u\right.\\
 &\qquad \qquad \quad \left.
 +\int_s^t\left(\int|H|^qW\right)^{1-1/q}
       \left(\sum_i\int|\mathcal F_i|^qW\right)^{1/q}\dd u
 \right]\longrightarrow0.
\end{align*}
The first time integral is finite by
\eqref{eq:fixed-N-localization-majorant}; the second is finite by
H\"older's inequality in time and the weighted flux bounds. Apply
Young's inequality as in
\cite[proof of Lemma~3.6, estimates following (3.45)]{JiangWuShortTime},
with the constants in \eqref{eq:common-flux-constants},
to \eqref{eq:fixed-N-localized-energy-identity}, integrate on $[s,t]$,
and let $R\to\infty$. Dominated convergence applies to the endpoint
energies and the terms without gradients, and Fatou's lemma applies
to the nonnegative dissipation. This proves
\eqref{eq:common-flux-dissipation} in integrated form on every
$[s,t]\subset[0,T]$ at fixed $\delta$.

Equation \eqref{eq:common-mollified-forward-L1} implies
$U^\delta\to U$ in $C_tL^1$. The pointwise bound and
\eqref{eq:fixed-N-weighted-Lq-interpolation}, with
$\lambda_0<(q-1)\beta_T$ and $\alpha<(q-1)M_2$, imply
\begin{equation}\label{eq:fixed-N-field-weighted-convergence}
 \sup_{t\le T}\|U^\delta(t)-U(t)\|_{L^q(W(t))}^q
 \le C_{N,k,T}\sup_{t\le T}\|U^\delta(t)-U(t)\|_1\longrightarrow0.
\end{equation}
For the differences $H^\delta=G-U^\delta$ and $H=G-U$, the same bound
applies because $H^\delta-H=-(U^\delta-U)$. For each $i\le k$, we have
\begin{equation}\label{eq:fixed-N-field-flux-convergence}
 \begin{aligned}
 \|E_f^\delta(\cdot,x_i)U^\delta-E_f(\cdot,x_i)U\|_{L_t^qL^q(W)}
 &\le\|E_f^\delta-E_f\|_\infty
                         \|U^\delta\|_{L_t^qL^q(W)}\\
 &\quad+\|E_f\|_\infty
                         \|U^\delta-U\|_{L_t^qL^q(W)}
 \longrightarrow0.
 \end{aligned}
\end{equation}
The same interpolation estimate proves convergence of the weighted
integrals uniformly in time, including at the endpoints. For $q>2$, H\"older's inequality gives
\[
\begin{aligned}
 \bigl\||H^\delta|^{q/2}-|H|^{q/2}\bigr\|_{L_t^2L^2(W)}
 &\le\frac q2\left(
       \|H^\delta\|_{L_t^qL^q(W)}^{q/2-1}
       +\|H\|_{L_t^qL^q(W)}^{q/2-1}\right)
       \|H^\delta-H\|_{L_t^qL^q(W)}
 \longrightarrow0.
\end{aligned}
\]
For $q=2$, the same convergence follows from
$\bigl||H^\delta|-|H|\bigr|\le|H^\delta-H|$.
The dissipation bound and weak compactness therefore imply
\begin{equation*}
 \sup_\delta\int_0^T\!\int
       |\nabla_{V_k}(|H^\delta|^{q/2})|^2W<\infty
       \end{equation*}
       	and
       	\begin{equation*}
 \nabla_{V_k}(|H^\delta|^{q/2})
       \rightharpoonup\nabla_{V_k}(|H|^{q/2})
       \quad\text{in }L^2(W\,\dd t\dd Z_k).
\end{equation*}
In particular, for every $[s,t]\subset[0,T]$, we have
\[
 \int_s^t\!\int |\nabla_{V_k}(|H|^{q/2})|^2W
 \le\liminf_{\delta\to0}\int_s^t\!\int
                |\nabla_{V_k}(|H^\delta|^{q/2})|^2W.
\]
The energy inequality follows for the original field $E_f$.

For the comparison with $f_t^{\otimes k}$, we use
$
 \sum_{i=1}^k\|\nabla_{v_i}f_t^{\otimes k}\|_1
     =k\|\nabla_vf(t)\|_1.
$
The additional residual caused by field mollification therefore satisfies
\begin{equation*}
 \left\|\frac{N-k}{N}\sum_i(E_f^\delta-E_f)(t,x_i)
                  \cdot\nabla_{v_i}f_t^{\otimes k}\right\|_{L_t^1L^1}
 \le k\|E_f^\delta-E_f\|_\infty
               \int_0^T\|\nabla_vf(t)\|_1\,\dd t\longrightarrow0.
\end{equation*}

\subsection{Limit of the integrated energy inequality}

We state the weighted convergence argument from the proof of
\cite[Lemma~A.4]{JiangWuShortTime} for general densities and differences.

\begin{lemma}[Convergence of weighted integrals]
\label{lem:common-weighted-energy-convergence}
Fix $N,k$ and $q>1$. Let $F^{\varepsilon,n}\to F^\varepsilon\to F$ successively
in $C([0,T];L^1)$, and let nonnegative energies
$e^\varepsilon,e\ge k+|V_k|^2$ satisfy
\begin{equation}\label{eq:fixed-N-energy-convergence}
 e^\varepsilon=e\quad\text{on }\mathcal C_\varepsilon^c,
 \qquad
 \mathcal C_\varepsilon=\{\min_{i<j}d_{\mathcal D}(x_i,x_j)\le2\varepsilon\}.
\end{equation}
Set $\mathcal C_\varepsilon=\varnothing$ for $k=1$.
Assume the following bounds hold uniformly in time and approximation parameters:
\begin{equation}\label{eq:fixed-N-weighted-convergence-bounds}
\begin{aligned}
 |F^{\varepsilon,n}|,|F^\varepsilon|
 &\le C_{N,k}e^{-\beta e^\varepsilon}
               \prod_i\langle x_i\rangle^{-M_2},\\
 |F|&\le C_{N,k}e^{-\beta e}
               \prod_i\langle x_i\rangle^{-M_2}.
\end{aligned}
\end{equation}
On $\R^d$, take $M_2>d$ and $\alpha<(q-1)M_2$; on $\T^d$ omit
spatial factors. For continuous $0<\lambda\le\lambda_0<(q-1)\beta$,
the weighted integrals satisfy
\begin{equation*}
 \mathcal E^{\varepsilon,n}(t)
 :=\int|F^{\varepsilon,n}(t)|^qe^{\lambda(t)e^\varepsilon}\varpi_k
 \longrightarrow\mathcal E^\varepsilon(t)\longrightarrow\mathcal E(t)
\end{equation*}
uniformly on $[0,T]$, with the corresponding definitions of
$\mathcal E^\varepsilon$ and $\mathcal E$. These functions are continuous.
For sufficiently small $\varepsilon$, a constant $C_{N,k}$
independent of $\varepsilon,n$ satisfies
\begin{equation}\label{eq:fixed-N-quantitative-energy-convergence}
 \begin{aligned}
 \|\mathcal E^{\varepsilon,n}-\mathcal E^\varepsilon\|_{C([0,T])}
 &\le C_{N,k}\|F^{\varepsilon,n}-F^\varepsilon\|_{C([0,T];L^1)},\\
 \|\mathcal E^\varepsilon-\mathcal E\|_{C([0,T])}
 &\le C_{N,k}\left(
 \|F^\varepsilon-F\|_{C([0,T];L^1)}+\varepsilon^d\right).
 \end{aligned}
\end{equation}
\end{lemma}

\begin{proof}
The integrands are bounded by
\begin{equation}\label{eq:fixed-N-common-dominator}
 D_{N,k}(Z_k)=C_{N,k}^qe^{-(q\beta-\lambda_0)(k+|V_k|^2)}
                 \prod_i\langle x_i\rangle^{\alpha-qM_2}\in L^1.
\end{equation}
For $F_1,F_2$ satisfying \eqref{eq:fixed-N-weighted-convergence-bounds}
with the same energy, the estimate in
\cite[Lemma~A.4, proof of (A.19)]{JiangWuShortTime} implies
\[
 \int\bigl||F_1|^q-|F_2|^q\bigr|
             e^{\lambda e^\varepsilon}\varpi_k
 \le C_{N,k}\|F_1-F_2\|_1.
\]
The same pointwise bounds also imply
\begin{equation}\label{eq:fixed-N-weighted-Lq-interpolation}
 \int|F_1-F_2|^qe^{\lambda e^\varepsilon}\varpi_k
 \le C_{N,k}\|F_1-F_2\|_1.
\end{equation}
These inequalities prove the first estimate in
\eqref{eq:fixed-N-quantitative-energy-convergence}.
Apply the same estimate on $\mathcal C_\varepsilon^c$, where
\eqref{eq:fixed-N-energy-convergence} holds, and use
\eqref{eq:fixed-N-common-dominator} on $\mathcal C_\varepsilon$:
\begin{equation}\label{eq:fixed-N-energy-epsilon-passage}
 \sup_{t\le T}|\mathcal E^\varepsilon(t)-\mathcal E(t)|
 \le C_{N,k}\sup_{t\le T}\|F^\varepsilon(t)-F(t)\|_1
                 +2\int_{\mathcal C_\varepsilon}D_{N,k}\longrightarrow0.
\end{equation}
Using $\alpha-qM_2<-M_2<-d$ on $\R^d$, the estimate on
$\mathcal C_\varepsilon$ in \cite[Lemma~A.4, (A.20)]{JiangWuShortTime}
yields
\[
 \int_{\mathcal C_\varepsilon}D_{N,k}
 \le C_{N,k}\varepsilon^d.
\]
Together with \eqref{eq:fixed-N-energy-epsilon-passage}, this proves the
second estimate in \eqref{eq:fixed-N-quantitative-energy-convergence}.
For continuity in time,
$u e^{-(q\beta-\lambda_0)u}\le C e^{-(q\beta-\lambda_0)u/2}$ and
$e^\varepsilon\ge k+|V_k|^2$ imply
\[
 \sup_{t\le T}\int e^\varepsilon|F^\varepsilon(t)|^q
              e^{\lambda_0e^\varepsilon}\varpi_k\le C_{N,k}.
\]
Hence, for $0\le s,t\le T$, we have
\begin{align*}
 |\mathcal E^\varepsilon(t)-\mathcal E^\varepsilon(s)|
 &\le q\int
 (|F^\varepsilon(t)|+|F^\varepsilon(s)|)^{q-1}
 |F^\varepsilon(t)-F^\varepsilon(s)|
 e^{\lambda(t)e^\varepsilon}\varpi_k\\
 &\quad+|\lambda(t)-\lambda(s)|
       \int e^\varepsilon|F^\varepsilon(s)|^q
             e^{\lambda_0e^\varepsilon}\varpi_k\\
 &\le C_{N,k}\bigl(\|F^\varepsilon(t)-F^\varepsilon(s)\|_1
                         +|\lambda(t)-\lambda(s)|\bigr).
\end{align*}
The same estimate applies to $\mathcal E^{\varepsilon,n}$.
Uniform convergence proves continuity of $\mathcal E$.
\end{proof}

\begin{remark}\label{rem:weighted-endpoints}
For $c\in L^1(0,T)$, we have
\[
 \left|\int_s^t c(r)(\mathcal E^\varepsilon-\mathcal E)(r)\,\dd r\right|
 \le\|c\|_{L^1(0,T)}\sup_{r\le T}|\mathcal E^\varepsilon(r)-\mathcal E(r)|.
\]
Thus integrated inequalities with fixed integrable coefficients and
convergent $L^1$ sources pass on every $[s,t]\subset[0,T]$,
including the two endpoint values.
\end{remark}

\subsection{Approximation with a bounded configuration field}
\label{subsec:ft-approximation}

Use the initial density $G_N^0$ and the bound
\eqref{eq:comparison-initial-envelope} from
Section~\ref{sec:common-comparison}, together with the spatial weights
in \eqref{eq:TV-alpha-choice}. On $\T^d$, omit the spatial factors.
The construction below uses the particle dynamics of
Section~\ref{sec:kernel-verification} and the pointwise estimates of
\Cref{lem:common-exponential-envelope}. The localization and weighted
convergence arguments then justify the tests and limits in part~(iii)
of \Cref{lem:ft-approximation}. The scalar inequalities mentioned
there are derived in Section~\ref{sec:fixed-time-riesz} for the
regularized equations.
In this subsection,
$b_i^N\in C_b^\infty(\mathcal D^N;\R^d)$ are permutation-equivariant
configuration fields with $\sup_{N,i}\|b_i^N\|_\infty\le B_0$. Their
derivative bounds may depend on $N$ and on the parameter $a$ below.

Use the decomposition \eqref{eq:ft-potential-split}, with
$0<a\le a_2$ and
$0\le\widetilde\phi^a\le\phi$,
$\overline\phi^a\ge0$, $\overline\phi^a\in L^\infty$ as in
\eqref{eq:ft-split-properties}.

Choose $\eta\in C^\infty([0,\infty);[0,1])$ equal to zero on $[0,1]$
and one on $[2,\infty)$. For $0<h<a/10$, define
\begin{equation*}
 \widetilde\phi^{a,h}(x)=c_{d,s}s\int_{|x|}^\infty
             \chi(r/a)\eta(r/h)r^{-s-1}\,\dd r,\qquad
 \phi^h=\widetilde\phi^{a,h}+\overline\phi^a.
\end{equation*}
Then $0\le\widetilde\phi^{a,h}\le\widetilde\phi^a$,
$0\le\phi^h\le\phi$, and both regularized potentials agree with
their singular counterparts for $|x|\ge2h$. Their forces satisfy
\[
 \widetilde K^{a,h}:=-\nabla\widetilde\phi^{a,h}
 =\eta(|x|/h)\widetilde K^a,\qquad
 K^h=\widetilde K^{a,h}+\overline K^a.
\]
Since $2h<a$, we have
\begin{equation}\label{eq:ft-core-kernel-difference}
 K^h(x)-K(x)=
 \begin{cases}
 c_{d,s}s\bigl(\eta(|x|/h)-1\bigr)\dfrac{x}{|x|^{s+2}},
       &0<|x|<2h,\\
 0,&d_{\mathcal D}(x,0)\ge2h.
 \end{cases}
\end{equation}
Radial integration and \eqref{eq:ft-near-kernel-constant} yield
\begin{equation}\label{eq:ft-core-regularization-bounds}
 \begin{aligned}
 &\|\widetilde K^{a,h}\|_{q,\varpi}^q
 \le\|\widetilde K^{a,h}\|_{L^{q'}}^q
 \le C_3a^{q(d-s-1)-d},\\
 &\|K^h-K\|_1\le C_{d,s}h^{d-s-1},\\
 &\|K^h-K\|_{L^{q'}}^q\le C_{d,s,q}h^{q(d-s-1)-d}.
 \end{aligned}
\end{equation}
The constants are independent of $a,h,N$.

Since $\phi^h\le\phi$ and $\beta_1<\beta_2$,
\eqref{eq:comparison-initial-envelope} implies the initial bound with
energy $e_{N,k}^h$ and exponent $\beta_1$ required by
\Cref{lem:common-smooth-initial-data,rem:smoothing-general-potential}.
These results provide smooth symmetric probability
densities $G_N^{0,h,n}$ such that
\begin{equation}\label{eq:ft-smooth-initial-envelope}
 G_{N,k}^{0,h,n}\le A_2^ke^{-\beta_0e_{N,k}^h}
                  \prod_{i=1}^kw_{M_2}(x_i),\qquad
 \xi_{n,h}:=\|G_N^{0,h,n}-G_N^0\|_1\longrightarrow0.
\end{equation}
The constant $A_2\ge1$ depends on $A,M_2,\beta_0,\beta_1$,
and is independent of $N,a,h,n$. The energies $\widetilde e_{N,k}^h$
and $e_{N,k}^h$ use $\widetilde\phi^{a,h}$ and $\phi^h$, respectively.
By $L^1$ contraction under marginalization and
\eqref{eq:conditioned-initial-error}, we have
\begin{equation*}
 \|G_{N,k}^{0,h,n}-f_0^{\otimes k}\|_1
 \le \xi_{n,h}+A^k/N.
\end{equation*}

Let $\widetilde G_N^{h,n}$ be the law with initial density
$G_N^{0,h,n}$ and acceleration
$N^{-1}\sum_{j\ne i}\widetilde K^{a,h}(X_i-X_j)+b_i^N(X_N)$.
Let $U_{N,k}^{h,n}$ solve the auxiliary equation with internal force
$K^h/N$, prescribed field $(N-k)E_f/N$, and initial datum
$G_{N,k}^{0,h,n}$. Put
$H_{N,k}^{h,n}=\widetilde G_{N,k}^{h,n}-U_{N,k}^{h,n}$.
Superscript $h$ alone denotes evolution from $G_N^0$; omission of $h$
means evolution with the corresponding singular force. In particular,
$H_{N,k}^{h,n}(0)=0$ for every $1\le k\le N$.

\begin{lemma}\label{lem:ft-approximation}
\begin{enumerate}[label=(\roman*),ref=(\roman*),leftmargin=*]
\item\label{item:ft-approximation-dynamics}
The singular systems have pathwise unique noncolliding solutions and
forward densities in $C([0,T];L^1)$. At fixed $N,a$, we have
\begin{equation}\label{eq:ft-approximation-L1}
 \begin{aligned}
 \widetilde G_{N,k}^{h,n}&\longrightarrow\widetilde G_{N,k}^h
                         \longrightarrow\widetilde G_{N,k},\\
 U_{N,k}^{h,n}&\longrightarrow U_{N,k}^h\longrightarrow U_{N,k}
 \end{aligned}
 \quad\text{in }C([0,T];L^1),
\end{equation}
first as $n\to\infty$, then as $h\downarrow0$.

\item\label{item:ft-approximation-weights}
With $\beta_T=\beta_0/(1+(4\sigma+2)\beta_0T)$, there are constants
$A_T,C_T\ge1$, independent of $N,a,h,n$, such that
\begin{equation}\label{eq:ft-approximation-envelopes}
 \begin{aligned}
 \widetilde G_N^{h,n}&\le A_T^Ne^{-\beta_T\widetilde e_{N,N}^h}
                              \prod_{i=1}^Nw_{M_2}(x_i),\\
 \widetilde G_{N,k}^{h,n}&\le C_T^Ne^{-\beta_T\widetilde e_{N,k}^h}
                              \prod_{i=1}^kw_{M_2}(x_i),\\
 U_{N,k}^{h,n}&\le A_T^ke^{-\beta_Te_{N,k}^h}
                              \prod_{i=1}^kw_{M_2}(x_i).
 \end{aligned}
\end{equation}
The bound with the full energy is
\begin{equation}\label{eq:ft-approximation-full-envelope}
 \widetilde G_{N,k}^{h,n}
 \le C_T^N e^{\beta_Tk\|\overline\phi^a\|_\infty}
                  e^{-\beta_Te_{N,k}^h}
                  \prod_{i=1}^kw_{M_2}(x_i).
\end{equation}
These estimates hold after both limits. For continuous
$0<\lambda(t)\le\lambda_0<(q-1)\beta_T$, the weighted integrals of the $q$-th powers of these densities and their
differences converge uniformly
on $[0,T]$, with the corresponding near-field or full energies.

\item\label{item:ft-approximation-limits}
For $\lambda'=-\Lambda\lambda^2$ as in
\Cref{lem:common-weighted-flux-energy}, the regularized marginal and
difference equations satisfy its energy inequality on every
$[s,t]\subset[0,T]$, with their internal near-field operator and
near-field weight. The singular marginal equations hold in distributions
with their initial traces.
For the field in \eqref{eq:ft-clipped-system}, the scalar inequalities
\eqref{eq:ft-smooth-near-energy} and
\eqref{eq:ft-regularized-error-hierarchy}, derived at fixed $N,a,h,n$,
pass successively to the limits $n\to\infty$ and $h\downarrow0$
on every $[s,t]\subset[0,T]$.
\end{enumerate}
\end{lemma}

\begin{proof}
\emph{(i).}
Apply \Cref{prop:common-finite-N,prop:common-regularization-removal}
with the near-field potential and field $b_i^N$ for the modified system,
and with the full potential and field $(N-k)E_f/N$ for the auxiliary
equation. These propositions establish noncollision, continuation and
convergence in $C_tL^1$ as $h\downarrow0$.
Markov contraction gives
\begin{equation*}
 \sup_{t\le T}\max\left\{
 \|\widetilde G_{N,k}^{h,n}-\widetilde G_{N,k}^h\|_1,
 \|U_{N,k}^{h,n}-U_{N,k}^h\|_1\right\}\le\xi_{n,h}.
\end{equation*}

\emph{(ii).}
\Cref{lem:common-exponential-envelope} applies with initial exponent
$\beta_0$ and drift bound $\max\{B_0,\|E_f\|_{C_tL^\infty}\}$.
This proves the first and third estimates in
\eqref{eq:ft-approximation-envelopes}. Integrating the first estimate with respect to the last
$N-k$ particle variables yields the second, by nonnegativity of the pair potential.
The energy comparison \eqref{eq:ft-full-near-energy} implies
\eqref{eq:ft-approximation-full-envelope}.
At each fixed time, extract an almost-everywhere convergent subsequence
from the $L^1$ convergence. Since the regularized potential converges
off the collision diagonals, which have measure zero, all pointwise
bounds hold for the limiting densities.

For each density or difference in
\eqref{eq:ft-approximation-L1}, the bounds
\eqref{eq:ft-approximation-envelopes}--\eqref{eq:ft-approximation-full-envelope}
satisfy \Cref{lem:common-weighted-energy-convergence} with the
corresponding near-field or full energy. At fixed $N,a$, its constant is
uniform in $h,n$. The weighted integrals therefore converge in
$C([0,T])$, first as $n\to\infty$, then as $h\downarrow0$.
In particular, \eqref{eq:fixed-N-quantitative-energy-convergence}
controls both endpoint values in every integrated inequality.

\emph{(iii).}
For the energy inequality, write $\widetilde J^h$ and $J^h$ for
integration against $\widetilde K^{a,h}$ and $K^h$.
With internal operator $\widetilde L_{N,k}^{a,h}$ from
\eqref{eq:ft-internal-far-force}, the marginal equation and the
equation for $H_{N,k}^{h,n}$ have the respective fluxes
\begin{equation}\label{eq:fixed-N-regularized-fluxes}
\begin{aligned}
 &\frac{N-k}{N}\widetilde J_{i,k}^h[\widetilde G_{N,k+1}^{h,n}]
            +\int b_i^N\widetilde G_N^{h,n}\,\dd Z_{k+1:N},\\
 &F_{i,k}^{a}H_{N,k}^{h,n}
 +\frac{N-k}{N}\left(J_{i,k}^h[\widetilde G_{N,k+1}^{h,n}]
                         -E_f(t,x_i)U_{N,k}^{h,n}\right)\\
 &\qquad+\int\left(b_i^N(X_N)
                -\frac1N\sum_{j\ne i}\overline K^a(x_i-x_j)\right)
                \widetilde G_N^{h,n}\,\dd Z_{k+1:N}.
\end{aligned}
\end{equation}
At $k=N$, the terms with coefficient $(N-k)/N$ are absent.
By \eqref{eq:ft-internal-far-force}, the internal far-field flux is
bounded by $M_a|H_{N,k}^{h,n}|$. Positivity gives
\begin{equation*}
 \left|\int b_i^N\widetilde G_N^{h,n}\,\dd Z_{k+1:N}\right|
       \le B_0\widetilde G_{N,k}^{h,n}
   \end{equation*}
   and
       	\begin{equation*}
 \left|\int\left(b_i^N(X_N)
       -\frac1N\sum_{j\ne i}\overline K^a(x_i-x_j)\right)
       \widetilde G_N^{h,n}\,\dd Z_{k+1:N}\right|
       \le(B_0+\|\overline K^a\|_\infty)\widetilde G_{N,k}^{h,n}.
\end{equation*}
Equations \eqref{eq:ft-approximation-envelopes} and
$e_{N,k}^h\ge\widetilde e_{N,k}^h$ bound both
$\widetilde G_{N,k}^{h,n}$ and $H_{N,k}^{h,n}$ by
$C_T^Ne^{-\beta_T\widetilde e_{N,k}^h}\prod_iw_{M_2}(x_i)$.
Together with \Cref{lem:common-weighted-holder}, these bounds imply
that each flux in \eqref{eq:fixed-N-regularized-fluxes} belongs to
$L^q(0,T;L^q(\widetilde W_{N,k}^h))$ at fixed $N,a,h,n$.
They also imply \eqref{eq:fixed-N-localization-majorant}.

At fixed $N,a,h,n$, the inverse-flow estimates in
\cite[Section~A.4, (A.11)]{JiangWuShortTime} also apply to
$\widetilde G_N^{h,n}$ since $b_i^N\in C_b^\infty$.
For every integer $r\ge0$, the product rule and marginalization yield
\[
 \left\|\int b_i^N\widetilde G_N^{h,n}\,\dd Z_{k+1:N}
             \right\|_{W^{r,1}}
 \le C_{r,N,k}\|b_i^N\|_{W^{r,\infty}}
                    \|\widetilde G_N^{h,n}\|_{W^{r,1}}.
\]
Together with the interaction-flux estimate in
\cite[Section~A.4, (A.12)]{JiangWuShortTime}, this proves local
smoothness of both fluxes in \eqref{eq:fixed-N-regularized-fluxes}
after replacing $E_f$ by $E_f^\delta$.

Apply Section~\ref{subsec:common-cutoff-testing} to these two fluxes
at fixed $N,a,h,n$. Equations
\eqref{eq:fixed-N-field-weighted-convergence}--\eqref{eq:fixed-N-field-flux-convergence}
allow us to remove the smoothing of $E_f$. The additional flux
converges in the same weighted space since
\[
 \|F_{i,k}^{a}(H^\delta-H)
                             \|_{L_t^qL^q(\widetilde W_{N,k}^h)}
 \le M_a\|H^\delta-H\|_{L_t^qL^q(\widetilde W_{N,k}^h)}
 \longrightarrow0.
\]
The integrated energy inequality therefore holds on every
$[s,t]\subset[0,T]$.

In the scalar inequalities \eqref{eq:ft-smooth-near-energy} and
\eqref{eq:ft-regularized-error-hierarchy}, derived for the regularized
equations at fixed $N,a,h,n$, the coefficient functions are independent
of $h,n$. The additional source satisfies
\begin{equation}\label{eq:fixed-N-error-source-convergence}
 \lim_{h\downarrow0}\lim_{n\to\infty}
 \|\xi_{n,h}+h^{q(d-s-1)}\|_{L^1(0,T)}=0,
 \qquad d-s-1>0.
\end{equation}
Uniform convergence of the weighted integrals and
\Cref{rem:weighted-endpoints} therefore allow us to pass successively
to the limits $n\to\infty$ and $h\downarrow0$ in these inequalities
on every $[s,t]\subset[0,T]$.

Finally, \eqref{eq:ft-core-regularization-bounds}, the pointwise
bounds, and \Cref{prop:common-singular-hierarchy} justify passage to
the limit in the internal and external interaction terms. Equation
\eqref{eq:bounded-configuration-flux-convergence} gives convergence
of the flux associated with the bounded configuration field. The
initial traces follow from \eqref{eq:ft-approximation-L1}.
\end{proof}

\section*{Statements and Declarations}
\textbf{Funding.} This work was supported in part by the National Natural
Science Foundation of China under Grant No.~12371180.\par
\textbf{Competing interests.} The authors declare that they have no competing
financial or non-financial interests.\par
\noindent\textbf{Data availability.} No datasets were generated or analyzed in the
course of this study.\par

\end{document}